\documentclass[11pt]{article}
\usepackage{amsthm,amsfonts,amssymb,amsmath}
\usepackage{graphicx}
\usepackage{floatrow}
\usepackage[arrow, matrix, curve]{xy}
\usepackage[margin=1.2in]{geometry}

\theoremstyle{plain}
\newtheorem{theorem}{Theorem}[section]

\newtheorem{corollary}[theorem]{Corollary}

\newtheorem{proposition}[theorem]{Proposition}
\newtheorem{example}[theorem]{Example}

\newtheorem{lemma}[theorem]{Lemma}
\newtheorem{definition}[theorem]{Definition}

\theoremstyle{remark}

\newcommand{\R}{{\mathbb{R}}}

\title{Morse Homotopy and Homological Conformal Field Theory}
\author{Viktor Fromm}
\begin{document}
\maketitle
\abstract{By studying spaces of flow graphs in a closed oriented manifold, we construct operations on its cohomology, parametrized by the homology of the moduli spaces of compact Riemann surfaces with boundary marked points. We show that the operations satisfy the gluing axiom of an open homological conformal field theory. This complements previous constructions due to R. Cohen et al., K. Costello and M. Kontsevich and is also the Morse theoretic counterpart to a conjectural construction of operations on the Lagrangian Floer homology of the zero section of a cotangent bundle, obtained by studying uncompactified moduli spaces of higher genus pseudoholomorphic curves. 
\section{Introduction}
A flow graph in a manifold $M$ consists of the following data: a graph $G$, a choice for every edge $e$ of $G$ of a flow $\Psi_{e}$ on $M$, and a continuous map $\boldsymbol{\gamma}: G \to M$ so that the image of each edge $e$ is a piece of a trajectory of the corresponding flow $\Psi_{e}$. Flow graphs can be used to obtain invariants of manifolds. The simplest instance of this is the Morse complex, corresponding to the case when $G$ consists of a single edge: by studying the spaces of trajectories of the gradient flow of a Morse function, one constructs a chain complex which computes the homology of a manifold. On the other hand, there are invariants which are not visible in the classical Morse complex but can be recovered using flow graphs (\cite{BeCo},\cite{CoNo},\cite{Fu}).\\ \\ In the work of R. Cohen and his collaborators (\cite{BeCo},\cite{CoNo}), flow graphs in manifolds were used to construct invariants in the form of cohomology operations (we remark that our terminology is somewhat different in that we use the term 'flow graph' instead of 'graph flow'). The operations associated to certain special graphs can be identified explicitly and turn out to correspond to invariants known from classical algebraic topology: the cup product, the Steenrod squares, the Stiefel-Whitney classes as well as the Massey product (the latter using a somewhat different approach) can all be encoded in this way. Moreover, the operations satisfy a field-theoretic law: there is a compatibility between gluing together graphs and composing the associated operations.\\ \\In this paper we construct, building upon seminal ideas of K. Fukaya (\cite{Fu}), the conformal version of this theory, obtained by studying flow graphs together with a ribbon structure on the underlying graph. By taking the ribbon structure into account in a suitable way, we define operations that are parametrized by the homology of the moduli spaces of Riemann surfaces with boundary. The construction draws upon ideas from Floer homology and Gromov-Witten theory.\\ \\
Let $\Sigma$ be a compact connected oriented surface with $m>0$ boundary components and $n_{+}+n_{-} \geq 0$ boundary marked points, partitioned into incoming and outgoing points. We assume that $2g-2+m>0$ and denote by $\mathcal{M}_{\Sigma}$ the space of complex structures on $\Sigma$, together with labellings of the marked points by positive real numbers. We will associate to every closed oriented manifold $M$ invariants in the form of linear maps \begin{equation} (H^{*}(M))^{\otimes n_{+}} \rightarrow H^{*}(\mathcal{M}_{\Sigma}) \otimes (H^{*}(M))^{\otimes n_{-}}. \label{EqHomOp} \end{equation}
Here $H^{*}$ denotes singular cohomology with coefficients in a field of characteristic zero (more precisely, for $\mathcal{M}_{\Sigma}$ cohomology with coefficients in a certain local system must be used). For example, if $\Sigma$ is a disk with two incoming and one outgoing marked point, then $\mathcal{M}_{\Sigma}$ is contractible and (\ref{EqHomOp}) is the cup product. The construction of the operations (\ref{EqHomOp}) in the general case proceeds as follows.\\ \\
Let $g$ be a Riemannian metric on $M$ and $f$ a Morse function whose gradient flow with respect to $g$ is Morse-Smale. Denote by $C^{*}(f)$ the associated Morse complex, with the grading given by the Morse index and the codifferential given by counts of trajectories of the positive gradient flow. In order to define the operations (\ref{EqHomOp}), it suffices to construct cochain maps
\begin{equation}F^{M}_{\Sigma}\colon  (C^{*}(f))^{\otimes n_{+}}  \rightarrow C^{*}(\mathcal{M}_{\Sigma}) \otimes (C^{*}(f))^{\otimes n_{-}}, \label{EqMorseOp} \end{equation} 
where $C^{*}(\mathcal{M}_{\Sigma})$ is a complex computing the cohomology of $\mathcal{M}_{\Sigma}$. To this end, the ribbon graph decomposition of Riemann surfaces (\cite{Ha},\cite{Pe}) will be used.\\ \\
A graph is a one-dimensional CW complex. We refer to the univalent vertices of a graph as the external vertices and to the vertices of valency greater than one as internal. An edge is called external if it is incident to a univalent vertex, and internal otherwise. Consider a graph $G$ together with an embedding $i: G \hookrightarrow \Sigma$, so that the univalent vertices of $G$ are mapped to the boundary marked points, the remaining points of $G$ are mapped to the interior of $\Sigma$ and so that $\Sigma$ deformation retracts to $i(G)$. Two embeddings are identified if one is obtained from the other by an isotopy which is constant on the univalent vertices. A {\it ribbon structure} on $G$ is an equivalence class of embeddings. We write $\Gamma$ for the ribbon graph and we say that $\Gamma$ has $\Sigma$ as its associated oriented surface.\\ 
\begin{figure}[!htbp]
\centering	\floatbox[{\capbeside\thisfloatsetup{capbesideposition={left,center},capbesidewidth=5cm}}]{figure}[\FBwidth]{\caption{Two ribbon graphs whose associated surfaces are a pair of pants and a torus with a disk removed respectively.}	\label{Fig1}} {\includegraphics[width=0.36\textwidth]{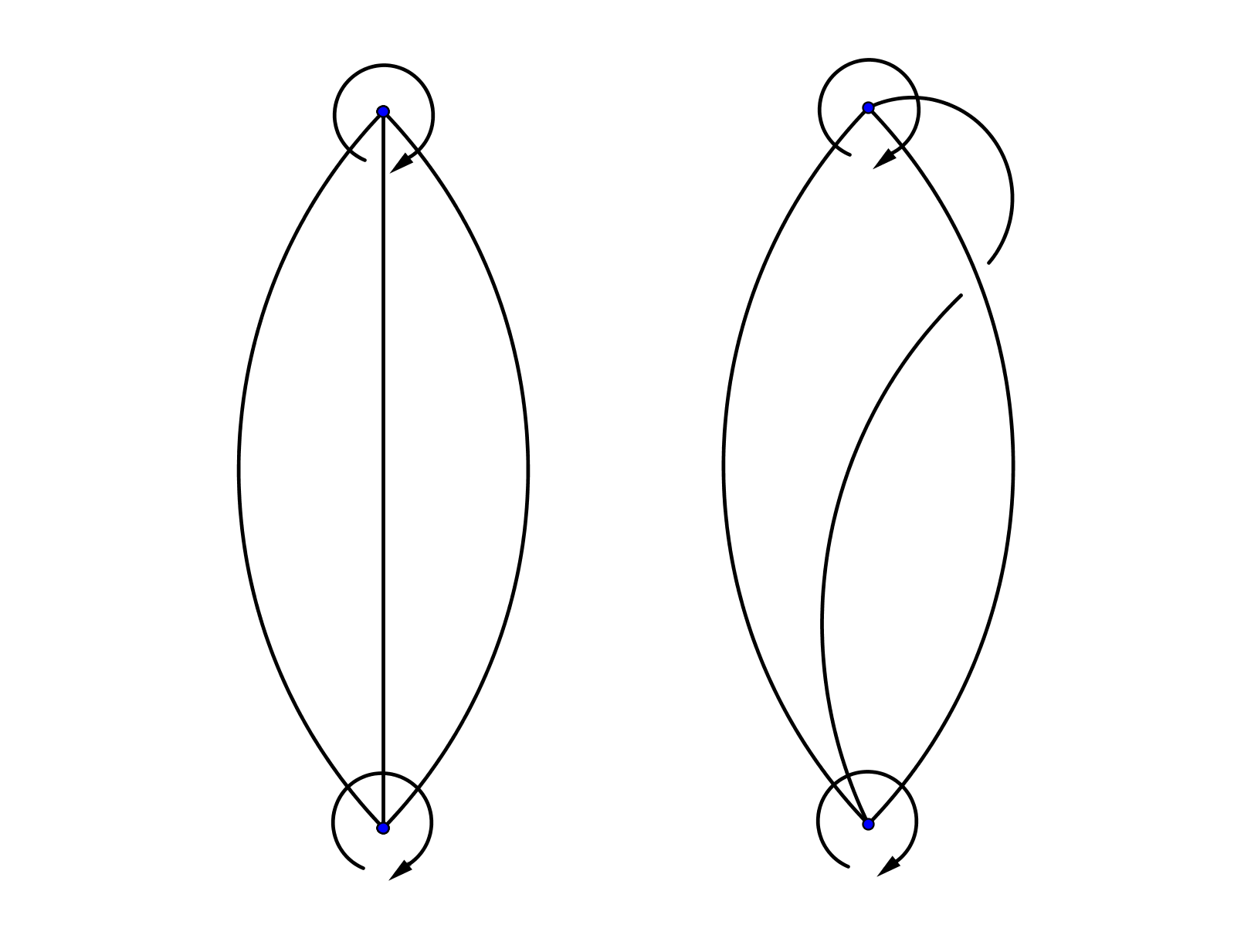}}
\end{figure} \\
A {\it half-edge} of a graph $G$ is a pair $(v,[\varphi])$, where $v$ is a vertex and $[\varphi]$ is an isotopy class of embeddings $\varphi: ([0,1],0) \hookrightarrow (G,v)$. It is a classical observation that a ribbon structure on $G$ can equivalently be defined as a cyclic ordering of the half-edges at every vertex.  An {\it orientation} of an edge is defined as an ordering of the two-element set consisting of the corresponding half-edges. A {\it metric structure} on $G$ is an assignment to every edge $e \in E(G)$ of a non-negative real number $l_{e}$.\\ \\
Assume that for every marked point $o \in \partial \Sigma$, a critical point $p_{o}$ of $f$ is fixed. In Section \ref{SecRG}, we associate to every ribbon graph $\Gamma$ as above a space $\mathcal{M}_{\Gamma,{\bf x}}({\bf p}_{+},{\bf p}_{-})$ of flow graphs in $M$. An element of $\mathcal{M}_{\Gamma,{\bf x}}({\bf p}_{+},{\bf p}_{-})$ consists of a metric structure on $\Gamma$ together with a continuous map $\boldsymbol{\gamma}: \Gamma \rightarrow M$, so that the restrictions of the map to the edges are pieces of trajectories of given vector fields ${\bf x}$ on $M$ and where incidence conditions corresponding to the points ${\bf p}_{+}, {\bf p}_{-}$ are imposed at the external vertices. For generic choice of the vector field datum {\bf x}, the space $\mathcal{M}_{\Gamma,{\bf x}}({\bf p}_{+},{\bf p}_{-})$ is a smooth manifold. Moreover, the partial compactification $\overline{\mathcal{M}}_{\Gamma,{\bf x}}({\bf p}_{+}, {\bf p}_{-})$ obtained by allowing internal edges of zero length as well as broken flow lines at the external edges is a manifold with corners. By the ribbon graph decomposition of Riemann surfaces, the space of metric ribbon graphs $\Gamma$ as above is homeomorphic to $\mathcal{M}_{\Sigma}$ (see equation (\ref{EqRibbonGrDec}) below) and thus there is a projection $\pi_{\Gamma}: \overline{\mathcal{M}}_{\Gamma,{\bf x}}({\bf p}_{+}, {\bf p}_{-}) \rightarrow \mathcal{M}_{\Sigma}$ obtained by forgetting $\boldsymbol{\gamma}$. We show that $\pi_{\Gamma}$ is a proper map. These claims are established in Section \ref{SubSecRGFlows}.\\ \\ We view $Z_{\Gamma,{\bf x}}({\bf p}_{+}, {\bf p}_{-})=(\overline{\mathcal{M}}_{\Gamma,{\bf x}}({\bf p}_{+}, {\bf p}_{-}),\pi_{\Gamma})$ as a geometric chain in $\mathcal{M}_{\Sigma}$. In Section \ref{SubSecConv} we define a chain complex $C^{BM}_{*}(\mathcal{M}_{\Sigma})$ generated by pairs $(P,f)$, where $P$ is a (not necessarily compact) oriented manifold with corners and $f: P \rightarrow \mathcal{M}_{\Sigma}$ is a proper continuous map. The complex $C^{BM}_{*}(\mathcal{M}_{\Sigma})$ computes the Borel-Moore homology of $\mathcal{M}_{\Sigma}$, the latter being isomorphic by rational Poincar\'e duality to the cohomology.  
\begin{theorem} \label{ThCohOp}
Let \begin{equation} F^{M}_{\Sigma}\colon  (C^{*}(f))^{\otimes n_{+}}  \rightarrow C^{BM}_{*}(\mathcal{M}_{\Sigma};det^{d} \otimes or) \otimes (C^{*}(f))^{\otimes n_{-}} \label{EqDefFSigma} \end{equation} be the linear map defined by
\begin{equation} {\bf p}_{+} \mapsto \sum \limits_{\Gamma, {\bf p}_{-}} Z_{\Gamma,{\bf x}}({\bf p}_{+}, {\bf p}_{-}) \otimes {\bf p}_{-},  \label{EqDefFSigmaCount} \end{equation}
where the summation is over all ${\bf p}_{-} \in (Crit(f))^{\times n_{-}}$ and over all the ribbon graphs $\Gamma$, so that every internal vertex of $\Gamma$ is trivalent. 
\begin{enumerate}
\item $F^{M}_{\Sigma}$ is a cochain map.
\item The induced map \begin{equation} H F^{M}_{\Sigma}: (H^{*}(M))^{\otimes n_{+}} \rightarrow H^{*}(\mathcal{M}_{\Sigma}; det^{\otimes d}) \otimes (H^{*}(M))^{\otimes n_{-}} \label{EqCohOp}\end{equation}  
is independent, up to a vector space isomorphism $H^{*}(M) \rightarrow H^{*}(M)$, of all choices (i. e. independent of the vector field data, of the Riemannian metric and of the Morse function). 
\end{enumerate}
\end{theorem} 
Here $d$ denotes the dimension of $M$ and $det$ and $or$ are certain local systems on $\mathcal{M}_{\Sigma}$, defined in Section \ref{SubSecOr}. The appearance of local coefficients is due to the fact that the space $\overline{\mathcal{M}}_{\Gamma,{\bf x}}({\bf p}_{+}, {\bf p}_{-})$ is not canonically oriented, but rather its orientation is determined by fixing linear orderings of the vertices and of the edges of $\Gamma$ and orientations of the edges. The dependence of the orientation of $\overline{\mathcal{M}}_{\Gamma,{\bf x}}({\bf p}_{+}, {\bf p}_{-})$ on these choices is described in Section \ref{SubSecOr}. Theorem \ref{ThCohOp} is proved in Section \ref{SubSecProofCohOp}.\\ \\
We remark that evaluating the statement of Theorem \ref{ThCohOp} in the case when $n_{+}=n_{-}=0$, i. e. there are no marked points, one obtaines invariants of closed oriented manifolds in the form of cohomology classes in moduli spaces of Riemann surfaces with boundary. However, examining the degrees, one finds that these classes are trivial unless the dimension of the manifold is at most three. \\ \\
We now outline the relationship of Theorem \ref{ThCohOp} to some previous results. As was mentioned above, a construction of operations from spaces of flow graphs was first proposed in \cite{BeCo} (see \cite{CoNo} for a more recent exposition). In this latter approach, one associates to a graph $G$ a cohomology operation, parametrized by the homology of the classifying space of the automorphism group $Aut(G)$. The main difference of the construction of Theorem \ref{ThCohOp} to these previous ideas is the consideration of ribbon structures. As a consequence, the operations constructed here are parametrized by the homology of the moduli spaces of Riemann surfaces instead. \\ \\
While the operations (\ref{EqCohOp}) bear some formal resemblence to the structure of Gromov-Witten invariants, there are two pronounced differences: firstly, $\Sigma$ has non-empty boundary and secondly, the operations are parametrized by the homology of $\mathcal{M}_{\Sigma}$ instead of the homology of the Deligne-Mumford compactification. The operations (\ref{EqCohOp}) fit into the framework of an open homological conformal field theory (\cite{Se},\cite{Ge}) - the corresponding gluing axiom is proved in Section \ref{SubSecGlue}. \\ \\
There are two previously known constructions of open HCFTs (in fact, more generally, of open TCFTs. i. e. {\it topological} conformal field theories). The first is an algebraic-combinatorial approach based on an idea of M. Kontsevich in \cite{Ko} (see also \cite{Cos1}, \cite{HamLaz}) and produces an open HCFT starting with a (finite-dimensional minimal) cyclic $A_{\infty}$-algebra - the homotopy associative analogue of a differential graded algebra, equipped with a compatible inner product. The second construction, due to K. Costello, is more analytic and uses heat kernels (\cite{Cos2}). The starting point here is a so-called Calabi-Yau elliptic space - one of the simplest examples of such an object is the de Rham algebra of a closed oriented manifold. The relationship between these two approaches is well-understood: a cyclic $A_{\infty}$-algebra can be associated to each Calabi-Yau elliptic space and the open HCFT obtained via the analytic construction on the elliptic space is equivalent to the result of applying the algebraic approach to the associated $A_{\infty}$-algebra (\cite{Cos2}, Section 5). \\ \\
Applying these ideas to the de Rham algebra, one obtains operations which have the same form as those constructed in Theorem \ref{ThCohOp} and which we denote by $H F^{dR}_{\Sigma}$. In a forthcoming paper, we will show that the constructions are equivalent, i. e. the operations $H F^{M}_{\Sigma}$ and $H F^{dR}_{\Sigma}$ coincide up to a vector space isomorphism $H^{*}(M) \rightarrow H^{*}(M)$. This explains how to compute the maps $H F^{M}_{\Sigma}$ from the de Rham algebra of the manifold with its product and inner product using the projection and homotopy which are provided (after choosing a Riemannian metric) by classical Hodge theory.\\ \\ 
It is known that for every open TCFT there is an associated universal open-closed, and thus in particular a closed TCFT (\cite{Cos1}). A direct geometric construction of an open-closed TCFT on a manifold, following ideas different from what is presented here, was outlined in \cite{BluCohTel}. The resulting closed TCFT is expected to be closely related to the BV algebra of loop homology (\cite{Ge}).\\ \\ 
We finish the introduction by sketching the relationship of the construction of Theorem \ref{ThCohOp} to Floer homology and the study of pseudoholomorphic curves in cotangent bundles. For a smooth function $f$ on a manifold $M$, the graph $L_{df}$ of the differential $df$ is an exact Lagrangian submanifold of the total space of the cotangent bundle $T^{*}M$. Two graphs $L_{df_{1}}$ and $L_{df_{2}}$ intersect transversally if and only if the difference $f_{2}-f_{1}$ is a Morse function. In this case the Lagrangian Floer cohomology of the pair $L_{df_{1}},L_{df_{2}}$ is defined (\cite{Fl}). The differential of Lagrangian Floer cohomology is constructed by counting the elements of the zero-dimensional components of the moduli spaces of pseudoholomorphic strips in $T^{*}M$ with $L_{df_{1}}$ and $L_{df_{2}}$ as boundary conditions. \\ \\ Extending this idea, we can consider for any oriented surface $\Sigma$ as in Theorem \ref{ThCohOp} the space of all pseudoholomorphic maps from $\Sigma$ to $T^{*}M$ with Lagrangian boundary conditions of the form $L_{df}$, $f \in \mathcal{C}^{\infty}(M,\R)$ imposed on the components of the complement in $\partial \Sigma$ of the set of marked points. For example, in the case when there are no marked points, we consider the space $\mathcal{M}^{T^{*}M}_{\Sigma}$ of all pseudoholomorphic curves in $T^{*}M$ which map every boundary component of $\Sigma$ to a submanifold $L_{df}$ for some fixed $f \in \mathcal{C}^{\infty}(M,\R)$. The space $\mathcal{M}^{T^{*}M}_{\Sigma}$ is usually non-compact, but if it carries a fundamental class in Borel-Moore homology, then using the fact that the projection $\pi_{\Sigma}: \mathcal{M}^{T^{*}M}_{\Sigma} \rightarrow \mathcal{M}_{\Sigma}$ is proper (this is a consequnce of Gromov compactness) and rational Poincar\'e duality, one would obtain a cohomology class in $\mathcal{M}_{\Sigma}$. More generally, by considering surfaces with marked points on the boundary, we would be led to operations on Lagrangian Floer cohomology, analogous to the ones constructed in Theorem \ref{ThCohOp}. If $\Sigma$ is a disk, this could be made rigorous by the methods developed in \cite{FOOO}. The general case is, to the knowledge of the author, conjectural.\\ \\
A. Floer showed that for a suitable choice of almost complex structures on $T^{*}M$ and of a Riemannian metric on $M$, the chain complex of Lagrangian Floer homology of a pair $L_{df_{1}}, L_{df_{2}}$ is isomorphic to the Morse complex of $f_{2}-f_{1}$ (\cite{Fl}). The main idea is sometimes referred to as an 'adiabatic limit' argument: Floer demonstrated that after multiplying $f_{1}$ and $f_{2}$ by a sufficiently small number, there is a one-to-one correspondence between isolated pseudoholomorphic strips in $T^{*}M$ and isolated gradient flow trajectories in $M$. In \cite{FOh}, this idea was generalized to obtain an identification between spaces of pseudoholomorphic disks with arbitrary number of boundary marked points and spaces of flows of ribbon trees was obtained. These results suggest that the operations on the Lagrangian Floer homology of the zero section of the cotangent bundle, defined as outlined above from the study of psudoholomorphic curves in $T^{*}M$, should correspond to operations on the cohomology of $M$, obtained from spaces of flows of general ribbon graphs. Theorem \ref{ThCohOp} provides the construction of these latter operations.\\ \\
Throughout this paper, homology and cohomology with real coefficients is used and the coefficient ring is omitted from the notation. Theorem \ref{ThCohOp} continues to hold for coefficients in an arbitrary field of characteristic zero.  \\ \\
The author would like to thank Klaus Mohnke for many fruitful conversations.        
\section{Ribbon Graphs and Morse Theory}\label{SecRG}
In this Section we discuss the complex $C^{BM}_{*}(\mathcal{M}_{\Sigma})$ of locally finite geometric chains in $\mathcal{M}_{\Sigma}$ and use flow graphs in a manifold to construct elements of this complex.    
\subsection{Geometric Chains and Borel-Moore Homology}\label{SubSecConv}
The Borel-Moore homology, or homology with closed support, of a topological space $X$ is defined as the homology of the chain complex of locally finite singular chains in $X$, i. e. of linear combinations $\sum n_{\sigma} \sigma$ of singular simplices $\sigma: \Delta \rightarrow X$, so that for every compact subset $K \subset X$ there are only finitely many non-vanishing coefficients $n_{\sigma}$ with $\sigma(\Delta) \cap K \neq \emptyset$. Summing the top-dimensional simplices of a triangulation of a (not necessarily compact) oriented manifold $N^{d}$, one defines its fundamental class $[N] \in H^{BM}_{d}(N)$. The map $x \mapsto [N] \cap x$ yields a Poincar\'e duality isomorphism \begin{equation} H^{k}(N) \xrightarrow{\sim} H^{BM}_{d-k}(N). \label{EqBMPDIsoGeneral} \end{equation} While Borel-Moore homology is not functorial in general, there is a pushforward $f_{*}: H^{BM}_{*}(X) \rightarrow H^{BM}_{*}(Y)$ provided that $f: X \rightarrow Y$ is proper. If $f: X \rightarrow Y$ is continuous and on both $X$ and $Y$ there are Poincar\'e duality isomorphisms as in (\ref{EqBMPDIsoGeneral}), then $f^{*}: H^{*}(Y) \rightarrow H^{*}(X)$ yields a transfer homomorphism $f^{!BM}_{*}: H^{BM}_{*}(Y) \rightarrow H^{BM}_{*+d_{X}-d_{Y}}(X)$ in Borel-Moore homology.\\ \\
Similarly to the case of singular homology discussed in \cite{Ja}, Borel-Moore homology may be computed as the homology of a complex of geometric chains. To define this complex, we recall that a $d$-dimensional manifold with corners $P$ is a paracompact Hausdorff space locally modelled on the products $\R^{k}_{\geq 0} \times \R^{d-k}$, $0 \leq k \leq d$. A local boundary component $\beta$ at a point $x$ of $P$ consists of a choice, for a coordinate neighbourhood $U$ of $x$, of a connected component of the points of $U$ which lie on the codimension one stratum. With the convention \begin{equation} \partial P = \{(x,\beta): x \in P, \beta \text{ a local boundary component of $P$ at $x$}\} \label{EqBdryCorners}, \end{equation} the boundary of a manifold with corners is again a manifold with corners and there is a natural map $i_{\partial P}: \partial P \rightarrow P$.\\ \\
We denote by $C^{BM}_{*}(X)$ the vector space of all sums of the form $\sum_{P}n_{P}(P,f_{P})$, where $P$ is an oriented manifold with corners and $f_{P}: P \rightarrow X$ is a proper continuous map. As before, we require that the geometric chain $\sum_{P}n_{P}(P,f_{P})$ be locally finite: for every $K \subset \subset X$, there are only finitely many summands so that $n_{P}$ is non-zero and $f_{P}(P) \cap K \neq \emptyset$. We make the following two identifications. Firstly, if $P'$ is obtained from $P$ by reversing the orientation, then we identify $(P',f_{P'})$ with $-(P,f_{P})$. Secondly, if $P$ is the union $P=Q \cup_{X} R$ of two codimension zero submanifolds with corners $Q$ and $R$ along a (possibly empty) submanifold with corners $X \subset \partial P, \partial Q$,  then we identify $(P,f_{P})$ with the sum of $(Q,f_{P}|_{Q})$ and $(R,f_{P}|_{R})$. The degree of $(P,f_{P})$ is the dimension of $P$ and the boundary is defined by $\partial (P,f_{P}) = (\partial P,f_{P} \circ i_{\partial P})$.
\begin{proposition} \label{PropIdBMHom}
For any topological space $X$, the homology of the complex $C^{BM}_{*}(X)$ is isomorphic to the Borel-Moore homology of $X$.
\end{proposition}   
\begin{proof}
Denote by $h^{BM}_{*}(X)$ the homology of the complex $C^{BM}_{*}(X)$ and by $H^{BM}_{*}(X)$ the Borel-Moore homology of $X$. By definition, the latter is given by equivalence classes of closed locally finite chains $\sum_{\sigma} n_{\sigma} \sigma$, where $\sigma: \Delta \rightarrow X$ is a singular simplex. A map $\psi: H^{BM}_{*}(X) \rightarrow h^{BM}_{*}(X)$ is defined by $[\sum_{\sigma} n_{\sigma} \sigma] \mapsto [\sum_{\sigma}n_{\sigma}(\Delta,\sigma)]$. Conversely, given a generator $(P,f_{P})$ of $C^{BM}_{*}(X)$, choose a triangulation of $P$ which induces a triangulation of $\partial P$ and so that if two components of $\partial P$ are diffeomorphic, then the induced triangulations coincide. Since $f$ is proper, the sum $\sum_{i}f_{P}|_{\Delta_{i}}=:\sum_{i}\sigma_{i}$ over the top-dimensional simplices of the triangulation is a locally finite singular chain in $X$; it is closed (resp. exact) if $(P,f_{P})$ is closed (resp. exact) and its class in $H^{BM}_{*}(X)$ is independent of the choice of the triangulation. Define $\phi: h^{BM}_{*}(X) \rightarrow H^{BM}_{*}(X)$ by $[P,f_{P}] \mapsto [\sum_{i}\sigma_{i}]$. It is immidiate to check that $\phi$ and $\psi$ are inverse to each other.     
\end{proof} 
In Section \ref{SubSecGlue}, we will use the fact that in two simple special cases the transfer homomorphism $f^{!BM}_{*}: H^{BM}_{*}(Y) \rightarrow H^{BM}_{*}(X)$ corresponding to a continuous map $f: X \rightarrow Y$ can be identified as the homomorphism induced by an explicit chain map $f^{!CBM}_{*}: C^{BM}_{*}(Y) \rightarrow C^{BM}_{*}(X)$. Firstly, if $f$ is the inclusion map $i: X \hookrightarrow Y$ of an open subset $X \subset Y$, then the image under $i^{!CBM}_{*}$ of a generator $(P,f_{P})$ of $C^{BM}_{*}(Y)$ is given by
$(f^{-1}_{P}(X),f_{P}|_{f^{-1}_{P}(X)})$. Secondly, if $f$ is the projection $\pi: X \rightarrow Y$ of a locally trivial fibre bundle, then the image of a generator $(P,f_{P})$ of $C^{BM}_{*}(Y)$ under $\pi^{!CBM}_{*}$ is given by $(f^{*}_{P}X,\pi^{*}f_{P})$, where $f^{*}_{P}X \rightarrow P$ is the pullback bundle and 
$\pi^{*}f_{P}$ the map which makes the following diagram commutative:\\
\begin{xy} \hspace{150pt} \xymatrix @C=0.45in{f^{*}_{P} X \ar[d]^/0em/{} \ar[r]^/0.2em/{\pi^{*}f_{P}} & X \ar[d]^/-0.2em/{\pi} \\ 
  P \ar[r]^/0em/{f_{P}} & Y. 
} \end{xy} \\ \\ \\
We will also consider the complex $C^{BM}_{*}$ with coefficients in a local system. In this context a generator is a pair $(P,f_{P})$ as before, together with a section of the pullback of the local system under $f_{P}$. If the local system is graded, then the degree of the chain is given as the sum of the dimension of $P$ and the degree of the section. There is a straightforward analogue of Proposition \ref{PropIdBMHom} for homology with local coefficients.\\ \\ 
As was mentioned above, on any oriented manifold there is a Poincar\'e duality isomorphism $H^{BM}_{*}(M) \simeq H^{ \text{dim }M-*}(M)$. Over a field of characteristic zero, this is also true for the space $\mathcal{M}_{\Sigma}$ of complex structures. To explain this fact, we recall that by the ribbon graph decomposition of Riemann surfaces (\cite{Ha},\cite{Pe}, see also \cite{Cos2}), there is a homeomorphism
\begin{equation} \mathcal{M}_{\Sigma} \simeq \bigcup \limits_{\Gamma} Met_{0}(\Gamma) / \sim. \label{EqRibbonGrDec} \end{equation}
The union on the right-hand side of (\ref{EqRibbonGrDec}) is over all ribbon graphs $\Gamma$ whose associated oriented surface is $\Sigma$ and so that every internal vertex of the graph has valency at least three; the univalent vertices of $\Gamma$ correspond to the marked points on $\partial \Sigma$. The symbol $Met_{0}(\Gamma)$ denotes the space of metric structures on $\Gamma$, where we allow edges of length zero, subject to the condition that the sum of the lengths of the edges in any cycle remain positive. The equivalence relation is generated by the following two identifications. Firstly, we identify a metric structure on a ribbon graph $\Gamma$ with each metric structure obtained as the pullback under a ribbon graph automorphism $\Gamma \rightarrow \Gamma$ which fixes each univalent vertex. Secondly, we identify two metric ribbon graphs if one is obtained from the other by collapsing internal edges of zero length. \\ \\
As a simple example, consider the case of a disk with four marked points on the boundary (for simplicity, the real-valued labellings of the marked points are omitted). 
\begin{figure}[!htbp]
\centering	\floatbox[{\capbeside\thisfloatsetup{capbesideposition={left,center},capbesidewidth=3.5cm}}]{figure}[\FBwidth]{\caption{The ribbon graph decomposition of the moduli space of holomorphic disks with four marked points on the boundary.}	\label{Fig2}} {\includegraphics[width=0.55\textwidth]{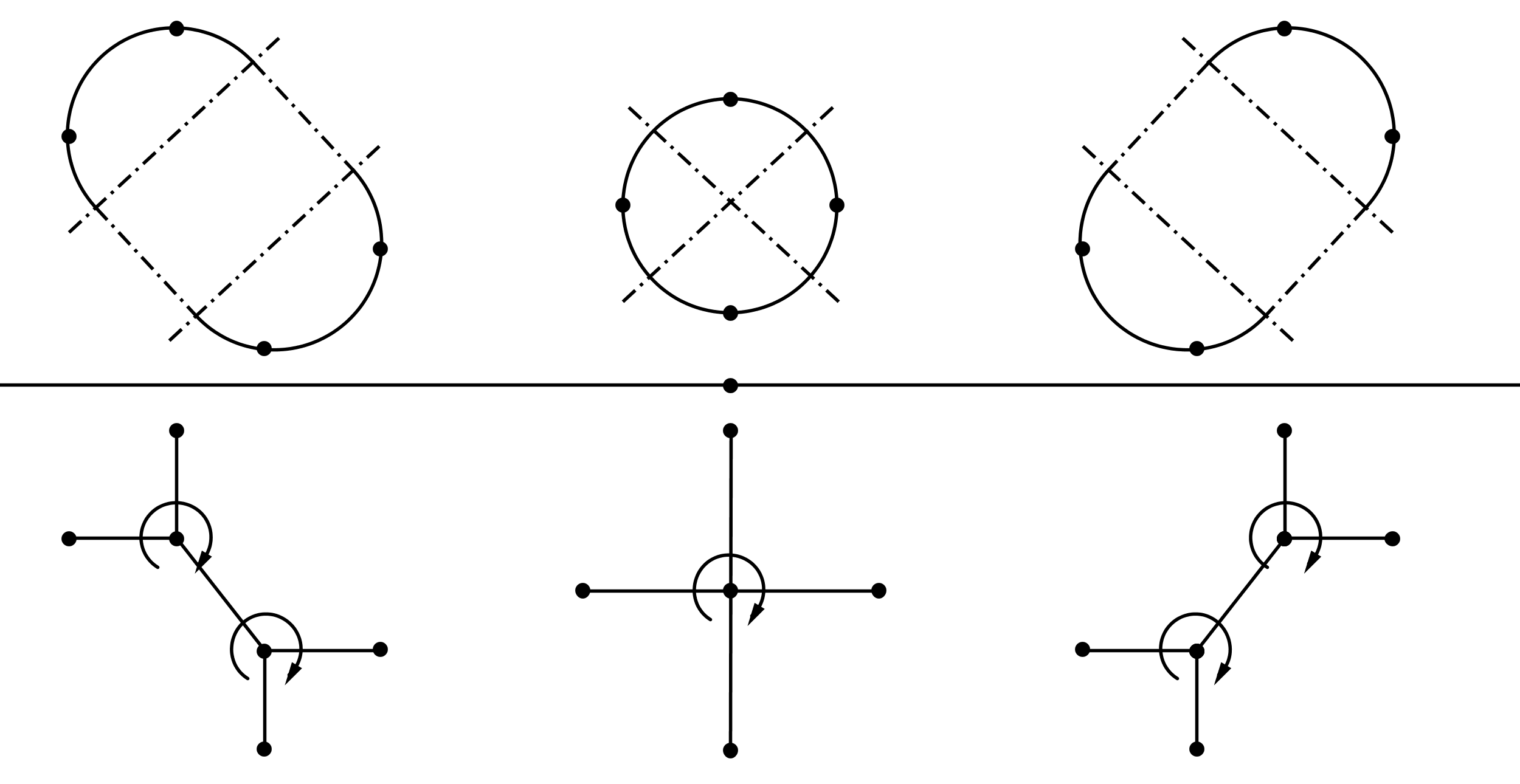}}
\end{figure} 
The moduli space is a real line which is identified as the result of gluing two copies of $\mathbb{R}_{\geq 0}$ to a point. The homeomorphism (\ref{EqRibbonGrDec}) means that each element of the moduli space is obtained uniquely from the fully symmetric configuration by slicing open along one of the two axes of symmetry which do not contain marked points, and then gluing in a holomorphic strip. \\ \\ 
It follows from (\ref{EqRibbonGrDec}) that $\mathcal{M}_{\Sigma}$ is an orbifold which admits a good cover in the sense of \cite{AdLeRu}, page 35. Thus over a field of characteristic zero, there is an isomorphism \begin{equation} H^{*}_{c}(\mathcal{M}_{\Sigma};or) \simeq H_{\text{dim } \mathcal{M}_{\Sigma}-*}(\mathcal{M}_{\Sigma}), \label{EqPDCsupp} \end{equation} where $H^{*}_{c}(\mathcal{M}_{\Sigma};or)$ is the cohomology with compact support and coefficients in the orientation sheaf.  By the universal coefficient theorem, $H^{*}_{c}(\mathcal{M}_{\Sigma};or)$ and $H_{*}(\mathcal{M}_{\Sigma})$ are the dual vector spaces to respectively $H^{BM}_{*}(\mathcal{M}_{\Sigma};or)$ and $H^{*}(\mathcal{M}_{\Sigma})$. We conclude the isomorphism \begin{equation} H^{BM}_{*}(\mathcal{M}_{\Sigma};or) \simeq H^{\text{dim } \mathcal{M}_{\Sigma}-*}(\mathcal{M}_{\Sigma}) \label{EqPDBM}.\end{equation}
\subsection{Flow Graphs in a Manifold}\label{SubSecRGFlows}  
We now introduce spaces of flow graphs in a manifold and use them to define elements of $C^{BM}_{*}(\mathcal{M}_{\Sigma})$. The discussion of transversality and the construction of natural compactifications of the spaces follows a similar line of argument as in the classical setting of Morse theory (\cite{Sch}).\\ \\
Let $\Gamma$ be a ribbon graph which appears on the right-hand side of the ribbon graph decomposition (\ref{EqRibbonGrDec}). Suppose that to every boundary marked point on $\partial \Sigma$ (or, equivalently, to every external edge $e$ of $\Gamma$), a critical point $p_{e}$ of the Morse function $f$ is associated. The partition of the marked points into incoming and outgoing points defines a partition ${\bf p}=({\bf p}_{+},{\bf p}_{-})$ of the tuple ${\bf p}$ of the critical points $p_{e}$. We will associate to this data a Banach manifold $\mathcal{B}_{\Gamma}({\bf p}_{+},{\bf p}_{-})$. We first introduce parametrizations of the edges of $\Gamma$. \\ \\ 
Let $e$ be an external edge of $\Gamma$ and $v$ be the external vertex incident to $e$. If $v$ is marked as incoming, then we fix a homeomorphism $\psi_{e}: (-\infty,0] \xrightarrow{\sim} e-\{v\}$. In the case when $v$ is marked as outgoing, we fix $\psi_{e}$ as a homeomorphism $\psi_{e}: [0,\infty) \xrightarrow{\sim} e-\{v\}$. In each case we orient $e$ by prescribing $\psi_{e}$ to be orientation-preserving. To choose parametrizations of the internal edges of $\Gamma$, we must take into account the fact that the choice of orientations of these edges is non-canonical. 
\begin{definition} We denote by $\Gamma' \rightarrow \Gamma$ the finite cover with fibre given by all choices of orientations of all the internal edges of $\Gamma$. 
\end{definition}
Thus $\Gamma'$ is a graph whose points are given by pairs consisting of a point of $\Gamma$ together with a choice of orientations of all the internal edges of $\Gamma$. Every edge of $\Gamma'$ carries a natural orientation and the projection $\pi_{\Gamma}: \Gamma' \rightarrow \Gamma$ preserves the orientations of the external edges. The group $T$ of covering transformations of $\Gamma' \rightarrow \Gamma$ is generated by the involutions $\tau_{e}$ given by reversing the orientation of an internal edge $e$ of $\Gamma$. \\ \\
We fix for every internal edge $e'$ of $\Gamma'$ a continuous map $\psi_{e'}: [0,1] \rightarrow e'$ whose restriction to $(0,1)$ is a homeomorphism and which induces the natural orientation of $e'$. If $e'$ is an incoming internal edge of $\Gamma'$ which is mapped to $e$ under $\pi_{\Gamma}$, then a map $\psi_{e'}: (-\infty,0] \rightarrow e'$ is uniquely defined by requiring that the composition of $\pi_{\Gamma} \circ \psi_{e'}$ coincide with $\psi_{e}$. We define parametrizations $\psi_{e'}: [0,\infty) \rightarrow e'$ of the outgoing edges of $\Gamma'$ analogously.\\ \\
If $e'$ is an incoming external edge of $\Gamma'$ and $p$ the associated critical point of $f$, then we denote by $H_{e'}$ the space of all elemenets $\gamma \in W^{1,2}_{loc}((-\infty,0],M)$ so that there exist $T > 0$ and $\xi \in W^{1,2}([T,\infty),T_{p}M)$ with $\gamma(t)=exp_{p}(\xi(t))$ for all $t \geq T$. Here $exp_{p}$ denotes the exponential map at $p$, defined in a neighbourhood of the origin of $T_{p}M$. In the case when $e'$ is outgoing, we define $H_{e'}$ analogously, but with $(-\infty,T]$ replaced by $[T,\infty)$. For an internal edge $e'$ of $\Gamma'$, we denote $H_{e'}=W^{1,2}([0,1],M)$.
\begin{definition} \label{DefB_Gamma}
We define $\mathcal{B}_{\Gamma}({\bf p}_{+},{\bf p}_{-})$ as the space of contunuous maps $\boldsymbol{\gamma}: \Gamma' \rightarrow M$, so that for every edge $e'$ of $\Gamma'$, the composition $\gamma_{e'}=\boldsymbol{\gamma} \circ \psi_{e'}$ is an element of $H_{e'}$ and moreover for any pair of edges $e',e''$ of $\Gamma'$, where $e'$ is an internal edge,
\begin{equation} \gamma_{\tau_{e'}(e'')}(t)=\begin{cases} \gamma_{e''}(t) & \text{ if }e' \neq e'',\\ \gamma_{e''}(1-t) & \text{ if }e'=e''. \label{EqB_Gamma} \end{cases} \end{equation}
\end{definition}
The space $\mathcal{B}_{\Gamma}({\bf p}_{+},{\bf p}_{-})$ is a Banach manifold. The tangent space $T_{\boldsymbol{\gamma}}\mathcal{B}_{\Gamma}({\bf p}_{+},{\bf p}_{-})$ at a point $\boldsymbol{\gamma}$ is the closed subspace of  $\otimes_{e' \in E(\Gamma')} W^{1,2}(\gamma^{*}_{e'}TM)$ consisting of all the elements ${\bf s}=(s_{e'})_{e' \in E(\Gamma')}$ which define a continuous section of $\boldsymbol{\gamma}^{*}TM$ and so that for each pair $e',e''$ of edges of $\Gamma'$, where $e'$ is an internal edge,  \begin{equation} s_{\tau_{e'}(e'')}(t)=\begin{cases} \;\; \; s_{e''}(t) & \text{ if }e' \neq e'',\\ -s_{e''}(1-t) & \text{ if }e'=e''. \end{cases} \label{EqTB_Gamma} \end{equation} Here $W^{1,2}(\gamma^{*}_{e}TM)$ denotes the space of sections of class $W^{1,2}$ and $E(\Gamma')$ is the set of edges of $\Gamma'$. \\ \\
 We define a Banach bundle $\mathcal{E} \rightarrow Met(\Gamma) \times \mathcal{B}_{\Gamma}({\bf p}_{+},{\bf p}_{-})$ as the pullback of $T \mathcal{B}_{\Gamma}({\bf p}_{+},{\bf p}_{-})$ under the projection $Met(\Gamma) \times \mathcal{B}_{\Gamma}({\bf p}_{+},{\bf p}_{-}) \rightarrow \mathcal{B}_{\Gamma}({\bf p}_{+},{\bf p}_{-})$ to the second factor. There is a well-defined section \begin{equation} {\bf S}=(s_{e'})_{e' \in E(\Gamma')} \in L^{2}(\mathcal{E}) \label{EqSectionS} \end{equation} determined by the condition that  \begin{equation} s_{e'}(\boldsymbol{\gamma})=\frac{d}{dt}\gamma_{e'}(t)\label{EqDefSectionS} \end{equation} for every edge $e'$ of $\Gamma'$.\\ \\
 We want to consider maps in $\mathcal{B}_{\Gamma}({\bf p}_{+},{\bf p}_{-})$, so that the restriction of the map to every edge is a piece of a trajectory of the flow of a given one-parameter family of vector fields. We now introduce the setup for the construction of these vector field data.\\ \\   
We choose for each edge $e'$ of $\Gamma'$ a one-parameter family of vector fields on $M$.  Formally, consider the vector bundle $E \rightarrow Met_{0}(\Gamma) \times [0,1] \times M$ given as the pullback of $TM$ under the projection to the last factor. Let $X_{\Gamma}$ denote the space of all sections of class $W^{1,2}$ of $E$. 
\begin{definition} \label{DefPertData}
We define $\mathcal{X}_{\Gamma}$ to be the space of all elements \begin{equation} \mathbf{x}=(x_{e'})_{e' \in E(\Gamma')} \in X_{\Gamma}^{\otimes {|E(\Gamma')|}} \end{equation} which satisfy the following conditions:
\begin{enumerate}
\item For any two edges $e',e''$ of $\Gamma'$, where $e'$ is an internal edge, and each $t\in [0,1]$,
\begin{equation} x_{\tau_{e'}(e'')}(\cdot, t,\cdot)=\begin{cases} \; \; \; x_{e''}(\cdot,t,\cdot) & \text{ if }e' \neq e'',\\ -x_{e''}(\cdot,1-t,\cdot) & \text{ if }e'=e''.
\end{cases} \end{equation}
\item There is a constant $C>0$, so that for every edge $e'$ of $\Gamma'$, the estimate
\begin{equation}\| x_{e'}(\textbf{\textit l}, t, \cdot) \|_{W^{1,2}(TM)}<C \label{EqCEstimate} \end{equation} holds true for all $(\textbf{\textit l},t) \in Met_{0}(\Gamma) \times [0,1]$.
\end{enumerate}
\end{definition} 
The first condition will be used to associate to every element of $\mathcal{X}_{\Gamma}$ a well-defined section of $\mathcal{E}$. The second condition is essential for the construction of compactifications of the spaces of flow graphs.\\ \\ 
To every element $\mathbf{x} \in \mathcal{X}_{\Gamma}$, we associate a section $\mathcal{F}_{\mathbf{x}}=(F_{e'})_{e' \in E(\Gamma')}$ of $\mathcal{E}$ as follows. Let $\sigma: \mathbb{R} \rightarrow \mathbb{R}$ be a smooth function with $\sigma(t)=1$ for $|t| \leq 1$ and $\sigma(t)=0$ for $|t| \geq 2$. If $e'$ is an external edge of $\Gamma'$, then \begin{equation} F_{e'}({\textbf{\textit l}},{\bf \gamma})(t)=\nabla_{g}f(\gamma_{e'}(t))+\sigma(t)x_{e'}(\textbf{\textit l},|t|,\gamma_{e'}(t)). \label{EqFExt} \end{equation} If $e'$ is an internal edge of $\Gamma'$ which is mapped to $e \in E(\Gamma)$ under the projection $\Gamma' \rightarrow \Gamma$, then we denote \begin{equation} F_{e'}({\textbf{\textit l}},{\bf \gamma})(t)=l_{e}x_{e'}(\textbf{\textit l},t,\gamma_{e'}(t)), \label{EqFInt} \end{equation}
where $l_{e}$ is the length of $e$ in the metric structure ${\textbf{\textit l}}$.\\ \\ 
With this notation in place, we can now define the spaces of flow graphs.
\begin{definition} \label{Def2}
For ${\bf x} \in \mathcal{X}_{\Gamma}$, let ${\bf S}_{{\bf x}} \in L^{2}(\mathcal{E})$ denote the section given as the difference \begin{equation} {\bf S}_{{\bf x}}={\bf S}-\mathcal{F}_{\mathbf{x}}, \label{EqSec1} \end{equation}
where ${\bf S}$ is defined as in (\ref{EqDefSectionS}). We define the space \begin{equation} \mathcal{M}_{\Gamma,{\bf x}}({\bf p}_{+},{\bf p}_{-}) \subset Met(\Gamma) \times \mathcal{B}_{\Gamma}({\bf p}_{+},{\bf p}_{-}) \label{EqDefGrFlow} \end{equation} of flows over $\Gamma$ subject to the vector field datum $\mathbf{x}$ as the zero locus of ${\bf S}_{{\bf x}}$.
 \end{definition}
Definition \ref{Def2} associates to each graph $\Gamma$ which appears in the ribbon graph decomposition (\ref{EqRibbonGrDec}) of $\mathcal{M}_{\Sigma}$ a corresponding space $\mathcal{M}_{\Gamma,{\bf x}}({\bf p}_{+},{\bf p}_{-})$ of flow graphs in a manifold. This is illustrated by the following simple example.  
 \begin{example} \label{Ex1}
 Consider the case when $\Sigma$ is an annulus with two marked points on the same boundary component. 
 \end{example}  
In this example, the moduli space $\mathcal{M}_{\Sigma}$ is homeomorphic to an open disk (for simplicity, we omit the real-valued labels at the marked points). There are five distinct isomorphism classes of ribbon graphs which appear on the right-hand side of (\ref{EqRibbonGrDec}). Three two-dimensional cells corresponding to the three ribbon graphs with two internal vertices of valency three are glued together along two one-dimensaion cells which correspond to the two graphs with a single internal vertex of valency four. There are no non-trivial automorphisms. The ribbon graph decomposition and the spaces of flow graphs associated to the cells are illustrated in Figure \ref{Fig3}. The shaded parts of the graphs indicate the vector field data.
\begin{figure}[!htbp]
\centering	\floatbox[{\capbeside\thisfloatsetup{capbesideposition={left,center},capbesidewidth=3cm}}]{figure}[\FBwidth]{\caption{Flow graphs corresponding to an annulus with two marked points on the same boundary component.}	\label{Fig3}} {\includegraphics[width=0.6\textwidth]{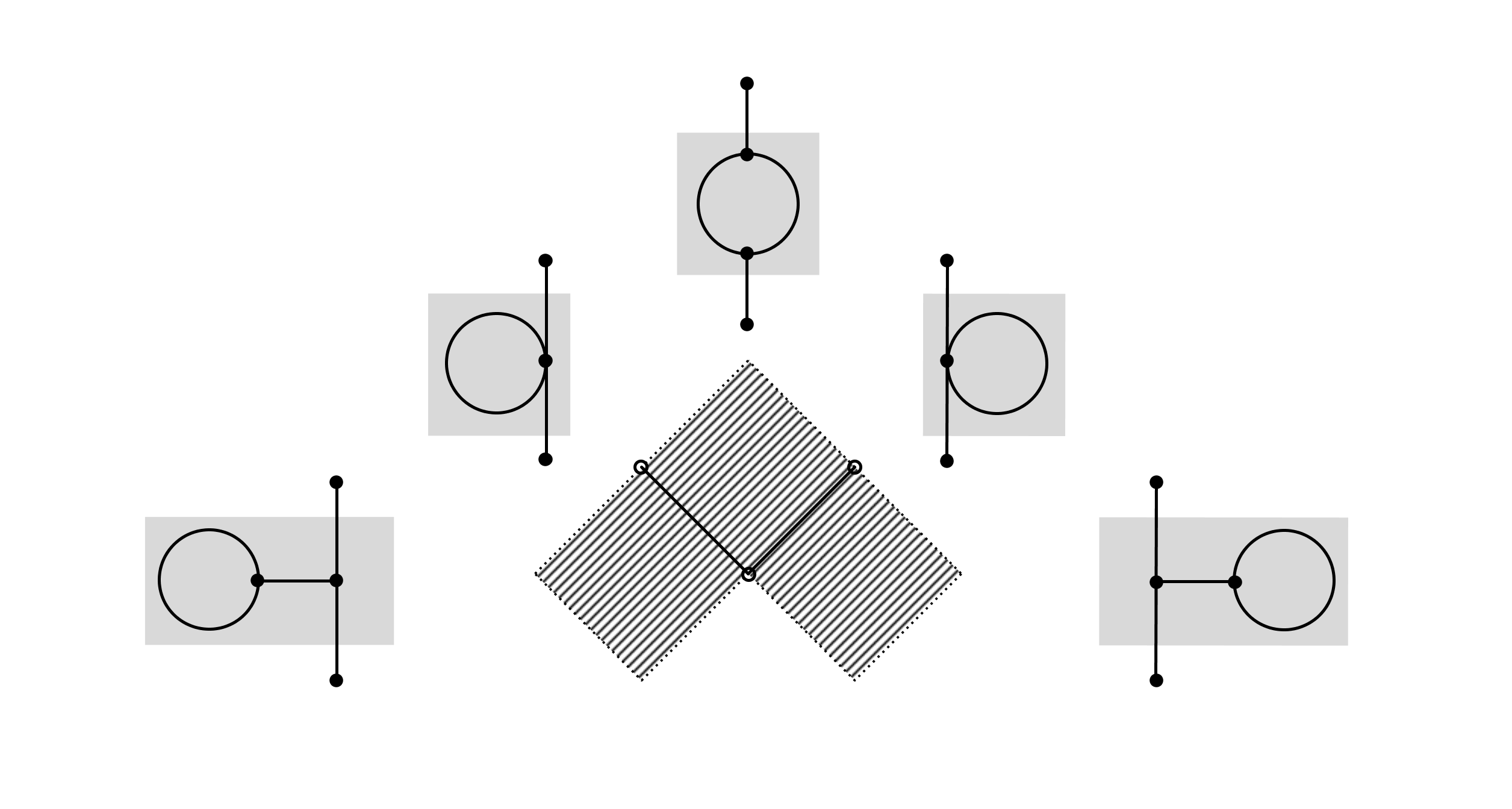}}
\end{figure} \\ \\
The next Proposition summarizes the arguments used to equip the space $\mathcal{M}_{\Gamma,{\bf x}}({\bf p}_{+},{\bf p}_{-})$ with the structure of a manifold.
\begin{proposition}\label{PropGener}
\begin{enumerate}
\item For every $\mathbf{x} \in \mathcal{X}_{\Gamma}$, the map \begin{equation} {\bf S}_{\mathbf{x}}: Met(\Gamma) \times \mathcal{B}_{\Gamma}({\bf p}_{+},{\bf p}_{-}) \rightarrow \mathcal{E} \label{EqMapSxFredh} \end{equation} is Fredholm of index
\begin{equation} ind({\bf S}_{\mathbf{x}})=|{\bf p}_{-}|-|{\bf p}_{+}| +d\chi(\Sigma)-d n_{-} +|E(\Gamma)|.\label{EqIndex} \end{equation}
Here $|{\bf p}_{+}|$ and $|{\bf p}_{-}|$ denote the sum of all the Morse indices of the critical points corresponding to the incoming and to the outgoing marked points respectively. The symbol $|E(\Gamma)|$ stands for the number of edges of $\Gamma$.
\item There is a subset $\mathcal{X}_{\Gamma, reg} \subset \mathcal{X}_{\Gamma}$ of second category so that for each $\mathbf{x} \in \mathcal{X}_{\Gamma, reg}$, ${\bf S}_{\mathbf{x}}$ is transverse to the zero section of $\mathcal{E}$.
\item Suppose that for every graph $\widetilde{\Gamma}$ obtained from $\Gamma$ by collapsing internal edges (where as before no cycle is collapsed), an element ${\bf x}_{\widetilde{\Gamma}} \in \mathcal{X}_{\widetilde{\Gamma}, reg}$ is fixed. Then there exists a subset of second category of $\mathcal{X}_{\Gamma}$, so that for every element ${\bf x}$ of that subset, the conclusion of 2. holds true and, in addition, the restriction of ${\bf x}$ to $Met_{0}(\widetilde{\Gamma}) \subset Met_{0}(\Gamma)$ coincides with ${\bf x}_{\widetilde{\Gamma}}$.
\end{enumerate}
\end{proposition} 
\begin{proof}
The arguments are analogous to the classical case of spaces of flow trajectories and we will thus stay brief. Denote by $D {\bf S}_{x}$ the linearization of ${\bf S}_{x}$. For fixed ${\textbf{\textit l}} \in Met(\Gamma)$ and ${\bf s} \in T_{{\bf \gamma}}\mathcal{B}_{\Gamma}({\bf p}_{+},{\bf p}_{-})$, we can write $D {\bf S}_{x}(0,{\bf s})$ in the form 
\begin{equation}(D {\bf S}_{x}(0,{\bf s}))_{e'}(t)= \frac{d}{dt}s_{e'}(t)-A(t)s_{e'}(t),\label{EqExprLin} \end{equation}
where $A(t) \in End(T_{\gamma_{e'}(t)}M)$ are endomorphisms such that if $e'$ is an external edge with $\gamma_{e'}(t) \rightarrow p \in Crit(f)$ for $|t| \rightarrow \infty$, then $A(t) \rightarrow Hess_{f}(p)$. Using the non-degeneracy of the Hessian, one concludes from (\ref{EqExprLin}) the inequality
\begin{equation}\|s_{e'}\|_{W^{1,2}} \leq c (\|s_{e'}\|_{L^{2}}+\|(D {\bf S}_{x}(0,{\bf s}))_{e'}\|_{L^{2}}) \label{EqLinEst} \end{equation}
for a positive constant $c$. Using the compactness of the embedding $W^{1,2}(T_{\bf \gamma}\mathcal{B}_{\Gamma}({\bf p}_{+},{\bf p}_{-})) \hookrightarrow L^{2}(T_{\bf \gamma}\mathcal{B}_{\Gamma}({\bf p}_{+},{\bf p}_{-}))$, it follows from (\ref{EqLinEst}) that the map $(D {\bf S}_{x})_{2}: {\bf s} \mapsto D {\bf S}_{x}(0,{\bf s})$ has finite-dimensional kernel and closed image. A standard computation using partial integration shows that each element ${\bf r} \in L^{2}(\mathcal{E})$ so that $\langle {\bf r}, D {\bf S}_{x}(0,{\bf s}) \rangle=0$ for all ${\bf s} \in T_{{\bf \gamma}}\mathcal{B}_{\Gamma}({\bf p}_{+},{\bf p}_{-})$ is weakly differentiable and satisfies \begin{equation} \frac{d}{dt}r_{e'}(t)+A^{T}(t)r_{e'}(t)=0 \label{EqCokernel}. \end{equation}
Together with the Sobolev embedding $W^{1,2}_{loc}(\R,M) \hookrightarrow C^{0}(\R,M)$ and using uniqueness of solutions of an ordinary differential equation, it follows that the cokernel of $(D {\bf S}_{x})_{2}$ is finite-dimensional and thus $(D {\bf S}_{x})_{2}$ is Fredholm. Since $Met(\Gamma)$ is finite-dimensional, we conclude that $D {\bf S}_{x}$ is Fredholm aswell. The index formula (\ref{EqIndex}) is straightforward.\\ \\
To prove the second part of the Proposition, we consider the map \begin{equation} \mathcal{S}_{\mathcal{X}}: \mathcal{X}_{\Gamma} \times Met(\Gamma) \times \mathcal{B}_{\Gamma}({\bf p}_{+},{\bf p}_{-}) \rightarrow \mathcal{E}, \\ ({\bf x},\textbf{\textit l},{\bf \gamma}) \mapsto \mathcal{S}_{\bf x}(\textbf{\textit l},{\bf \gamma}). \label{EqMaapSX} \end{equation} It suffices to show that $\mathcal{S}_{\mathcal{X}}$ is transverse to the zero section of $\mathcal{E}$. To this end, we must check that if $\mathcal{S}_{\mathcal{X}}(\mathbf{x},{\textbf{\textit l}},{\bf \gamma} )=0$, then every element ${\bf r} \in L^{2}(\mathcal{E})$, so that $\langle {\bf r},D \mathcal{S}_{\mathcal{X}}({\bf y},{\bf m},{\bf s}) \rangle=0$ for all $({\bf y},{\bf m},{\bf s}) \in \mathcal{X}_{\Gamma} \oplus T_{{\textbf{\textit l}}}Met(\Gamma) \oplus T_{\bf \gamma} \mathcal{B}_{\Gamma}({\bf p}_{+},{\bf p}_{-})$, vanishes identically. By the proof of the first part of the Proposition, each component $r_{e'}$, $e' \in E(\Gamma')$, is continuous. Thus it suffices to show that $r_{e'}(t)=0$ for $t \in (0,1)$. For simplicity, we will only carry this out in the case when $e'$ is an internal edge, the case when $e'$ is an external edge being similar. \\ \\
From (\ref{EqFInt}), we have \begin{equation} (D \mathcal{S}_{\mathcal{X}}({\bf y},0,0))_{e'}(t)=y_{e'}(t,\gamma_{e'}(t)). \label{EqCompDS} \end{equation} Suppose that $r_{e'}(t_{0}) \neq 0$, $t_{0} \in (0,1)$. Then there exist $0<\varepsilon < \min (t_{0},1-t_{0})$ and $y_{e'} \in X_{\Gamma}$, so that  \begin{equation} \langle r_{e'}(t),y_{e'}(\textbf{\textit l},t,\gamma_{e'}(t))\rangle_{g}>0 \text{ for }|t-t_{0}| < \varepsilon \label{EqCondy1} \end{equation} and
\begin{equation}  y_{e'}(\textbf{\textit l},t,\cdot) \equiv 0 \text{ for } |t-t_{0}| \geq \varepsilon. \label{EqCondy2} \end{equation} Using $y_{e'}$, we define ${\bf y} \in \mathcal{X}_{\Gamma}$ as follows. Let $e$ be the edge of $\Gamma$ corresponding to $e'$ under
 the projection $\Gamma' \rightarrow \Gamma$. If $e''$ is an edge of $\Gamma'$ obtained from $e'$ by a covering transformation of $\Gamma' \rightarrow \Gamma$ which preserves the orientation of $e$, then we define $y_{\tau_{e''}(e')}({\textbf{\textit l}},t,\cdot)=y_{e'}({\textbf{\textit l}},t,\cdot)$. In the case when $e''$ is obtained from $e'$ by applying a covering transformation which reverses the orientation of $e$, we put $y_{\tau_{e'}(e')}(\textbf{\textit l},t,\cdot)=-y_{e'}(\textbf{\textit l},1-t,\cdot)$. Finally, define $y_{e''} \equiv 0$ for all other edges $e''$ of $\Gamma'$. Then ${\bf y}=(y_{e'})_{e' \in E(\Gamma')}$ is a well-defined element of $\mathcal{X}_{\Gamma}$. Using (\ref{EqTB_Gamma}), we compute:
 \begin{equation} \langle {\bf r},D \mathcal{S} ({\bf y},0,0) \rangle  = 2|E_{int}(\Gamma)| \int \limits_{t=0}^{\varepsilon} \langle r_{e'}(t), y_{e'}(t,\gamma_{e'}(t)) \rangle_{g}>0 \end{equation}
in contradiction to the choice of ${\bf r}$. This completes the proof of the second part of the Proposition. The proof of the third part is analogous.
\end{proof}
\begin{corollary} \label{CorTransv}
\begin{enumerate}
\item For every element ${\bf x} \in \mathcal{X}_{\Gamma,reg}$, the space $\mathcal{M}_{\Gamma,{\bf x}}({\bf p}_{+},{\bf p}_{-})$ is a manifold whose dimension is given by the right-hand side of the index formula (\ref{EqIndex}). The space is empty in the case when the right-hand side of (\ref{EqIndex}) is negative. 
\item Let $\mathcal{X}_{\Sigma}$ denote the vector space 
\begin{equation} \mathcal{X}_{\Sigma}= \oplus_{\Gamma} \mathcal{X}_{\Gamma} \label{EqDefXSigma}, \end{equation} where the direct sum is over all the ribbon graphs $\Gamma$ in the ribbon graph decomposition (\ref{EqRibbonGrDec}) of $\mathcal{M}_{\Sigma}$.
There is a subset $\mathcal{X}_{\Sigma,reg} \subset \mathcal{X}_{\Sigma}$ of second category, so that \begin{equation} \mathcal{X}_{\Sigma,reg} \subset \oplus_{\Gamma} \mathcal{X}_{\Gamma,reg} \label{EqSigmaReg} \end{equation} and, in addition, every element $\bf{y}=(\bf{x}_{\Gamma})_{\Gamma} \in \mathcal{X}_{\Sigma,reg}$ satisfies the following condition: if the graph $\widetilde{\Gamma}$ is obtained from $\Gamma$ by collapsing internal edges, then $\bf{x}_{\widetilde{\Gamma}}$ coincides with the restriction of $\bf{x}_{\Gamma}$ 
to $Met_{0}(\widetilde{\Gamma}) \subset Met_{0}(\Gamma)$.
\end{enumerate}
\end{corollary}
The second part of this Corollary means that the vector field data for different ribbon graphs can be chosen consistently with the attachments of the corresponding cells in the ribbon graph deocmposition.  
\begin{proof}
The first part of the Corollary follows from the first two parts of Proposition \ref{PropGener}. To prove the second part, start by associating an arbitrary vector field datum in $\mathcal{X}_{\Gamma,reg}$ to each graph $\Gamma$ labelling a top-dimenional cell in (\ref{EqRibbonGrDec}) and use the third part of Proposition \ref{PropGener} successively to extend over the remaining cells.   
\end{proof}
We will consider certain partial compactifications of the spaces $\mathcal{M}_{\Gamma,{\bf x}}({\bf p}_{+},{\bf p}_{-})$. The main idea is as follows. We observe that there are the following three sources of non-compactness of the spaces:
\begin{enumerate}
\item Breaking of the trajectory corresponding to an external edge of the graph into several trajectories connecting critical points.
\item Convergence to zero of the length $l_{e}$ of an internal edge $e$ of the graph.
\item Breaking of a trajectory corresponding to an internal edge of the graph.
\end{enumerate}
\begin{figure}[!htbp]
\centering	\floatbox[{\capbeside\thisfloatsetup{capbesideposition={left,center},capbesidewidth=3cm}}]{figure}[\FBwidth]{\caption{Three types of boundary strata of natural compactifications of the spaces of flow graphs.}	\label{Fig4}} {\includegraphics[width=0.6\textwidth]{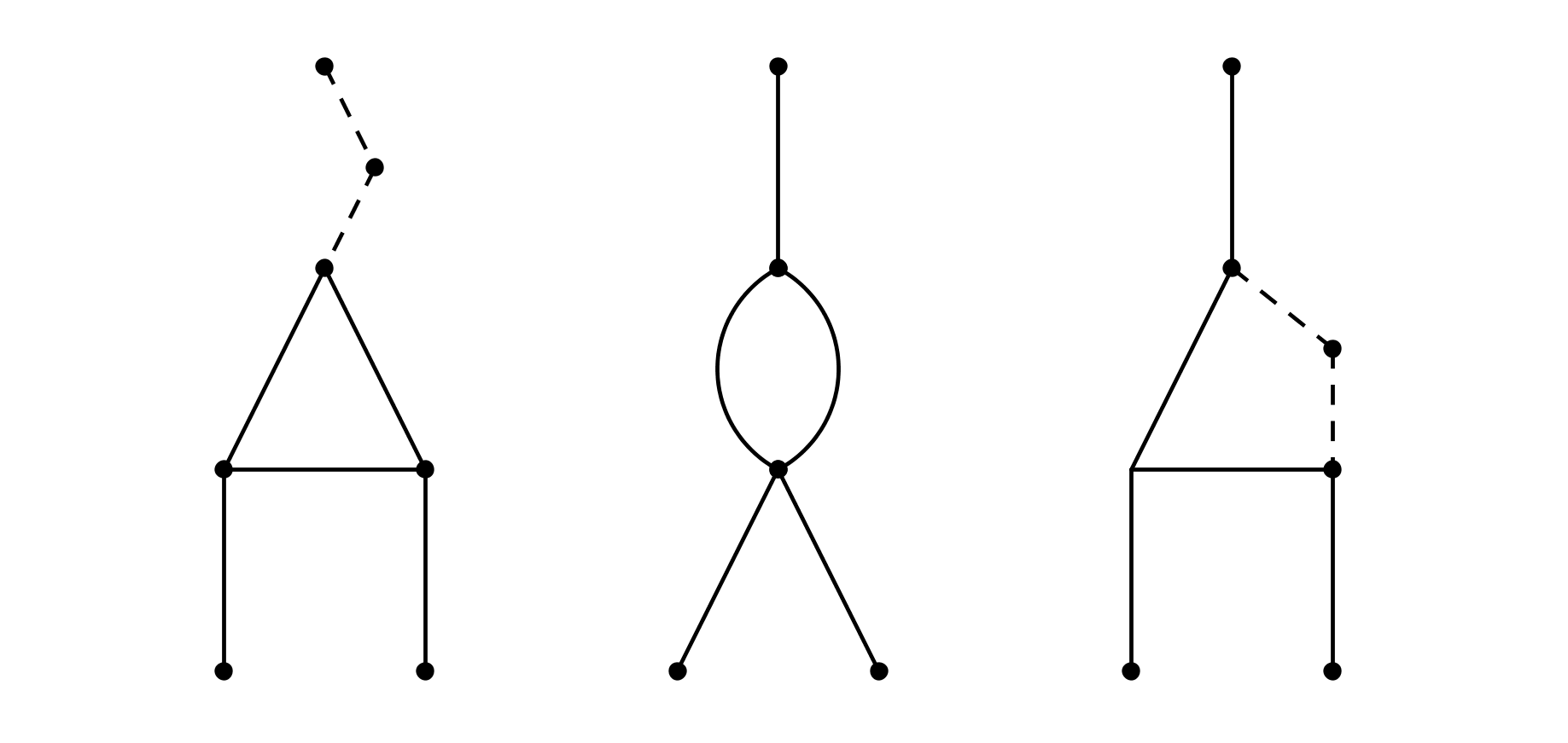}}
\end{figure} 
We will consider the partial compactification $\overline{\mathcal{M}}_{\Gamma,{\bf x}}({\bf p}_{+},{\bf p}_{-})$ of $\mathcal{M}_{\Gamma,{\bf x}}({\bf p}_{+},{\bf p}_{-})$ obtained by adding the strata of the first and the second (but not the third) type. We now give the formal definition.\\ \\
Given two critical points $p$ and $p'$ of $f$, denote by
\begin{equation} \mathcal{M}(p,p')=\{\gamma: \R \rightarrow M, \dot{\gamma}(t)=\nabla_{g}f(\gamma(t)), \lim_{t \rightarrow -\infty } \gamma(t)=p,\lim_{t \rightarrow \infty } \gamma(t)=p'\}/\R \label{EqMpq} \end{equation}
the space of flow trajectories emanating at $p$ and converging to $p'$. Denote by
\begin{equation} \overline{\mathcal{M}}(p,p')= \mathcal{M}(p,p') \cup \bigcup \limits_{m \geq 1; q_{1},\dots,q_{m}} \mathcal{M}(p,q_{1}) \times \dots \times \mathcal{M}(q_{m},p') \label{EqMbarpq} \end{equation}
the corresponding space of broken trajectories. Given two $n$-tuples ${\bf p}=(p_{1}, \dots, p_{n})$ and ${\bf p'}=(p'_{1}, \dots, p'_{n})$ of critical points, we write 
\begin{equation}
\overline{\mathcal{M}}({\bf p},{\bf p'})=\overline{\mathcal{M}}(p_{1},p_{1}') \times \dots \times \overline{\mathcal{M}}(p_{n},p_{n}'). \label{EqMbarpqbar} \end{equation}
\begin{definition} \label{DefMGammaxBar}
We define 
\begin{equation}
\overline{\mathcal{M}}_{\Gamma,{\bf x}}({\bf p}_{+},{\bf p}_{-})=
\bigcup \limits_{{\bf q}_{+}, {\bf q}_{-}, \widetilde{\Gamma} \prec \Gamma} \overline{\mathcal{M}}({\bf p}_{+},{\bf q}_{+}) \times \mathcal{M}_{\widetilde{\Gamma},{\bf x}}({\bf q}_{+},{\bf q}_{-}) \times \overline{\mathcal{M}}({\bf q}_{-},{\bf p}_{-}), \label{EqDefMGammabar}
\end{equation}
where the union is over all ${\bf q}_{+} \in Crit(f)^{\times n_{+}}$, ${\bf q}_{-} \in Crit(f)^{\times n_{-}}$ and all ribbon graphs $\widetilde{\Gamma}$ obtained by collapsing edges of $\Gamma$, so that no cycle is collapsed.
\end{definition}
We now establish the properties of the spaces $\overline{\mathcal{M}}_{\Gamma,{\bf x}}({\bf p}_{+},{\bf p}_{-})$ which will be used in the proof of Theorem \ref{ThCohOp}.
\begin{proposition} \label{PropMbarProperties}

\begin{enumerate}
\item For every ${\bf x} \in \mathcal{X}_{\Gamma}$ as in the third part of Proposition \ref{PropGener}, the space $\overline{\mathcal{M}}_{\Gamma,{\bf x}}({\bf p}_{+},{\bf p}_{-})$ is a manifold with corners.
\item The boundary $\partial \overline{\mathcal{M}}_{\Gamma,{\bf x}}({\bf p}_{+},{\bf p}_{-})$ is given by the disjoint union
\begin{equation} \begin{split} 
\partial \overline{\mathcal{M}}_{\Gamma,{\bf x}}({\bf p}_{+},{\bf p}_{-})=
&\quad \, \, \, (\coprod\limits_{|{\bf q}_{+}|-|{\bf p}_{+}|=1} \mathcal{M}({\bf p}_{+},{\bf q}_{+}) \times \overline{\mathcal{M}}_{\Gamma,{\bf x}}({\bf q}_{+},{\bf p}_{-}))\\  \coprod &\quad \, \, \, (\coprod\limits_{|{\bf p}_{-}|-|{\bf q}_{-}|=1}  \overline{\mathcal{M}}_{\Gamma,{\bf x}}({\bf q}_{+},{\bf p}_{-}) \times \mathcal{M}({\bf q}_{-},{\bf p}_{-}))\\  \coprod & \, \, \, \, \, \, (\coprod \limits_{e \in E_{int}(\Gamma) - L(\Gamma)} \overline{\mathcal{M}}_{\Gamma/e,{\bf x}}({\bf p}_{+},{\bf p}_{-})),  \label{EqMGammaBoundary}
\end{split}
\end{equation}
where the last union is over all the internal edges $e$ of $\Gamma$ which are not loops and where $\Gamma/e$ denotes the ribbon graph obtained from $\Gamma$ by collapsing $e$.
\item The projection $\pi_{\Gamma}: \overline{\mathcal{M}}_{\Gamma,{\bf x}}({\bf p}_{+},{\bf p}_{-}) \rightarrow \mathcal{M}_{\Sigma}$ defined by forgetting ${\bf \gamma}$ is proper. 
\end{enumerate}
\end{proposition}
\begin{proof}
In order to equip the space $\overline{\mathcal{M}}_{\Gamma,{\bf x}}({\bf p}_{+},{\bf p}_{-})$ with the structure of a manifold with corners, we will identify it as a transverse intersection of a manifold and a manifold with corners.\\ \\
Assume that a connected component $\Gamma'_{0}$ of $\Gamma'$ is fixed or, equivalently, an orientation of every internal edge of $\Gamma$ is chosen. For an internal edge $e'$ of $\Gamma'_{0}$, denote by $\Phi^{e'}_{{\textbf{\textit l},\bf{x}}}(t,\cdot): M \rightarrow M$ the flow of the vector field on the right-hand side of (\ref{EqFInt}). We define $M_{e'}=M_{e'}({\textbf{\textit l}},{\bf{x}}) \subset M \times M$ as the subspace
\begin{equation}M_{e'}=\{ (q,\Phi^{e'}_{{\textbf{\textit l},\bf{x}}}(1,q)) \in M \times M: q \in M \}. \label{EqDefMe} \end{equation}
Thus $M_{e'}$ is a manifold diffeomorphic to $M$. \\ \\
Suppose now that $e'$ is an external edge of $\Gamma'_{0}$ with the corresponding critical point $p_{e'}$. Denote by $\Phi^{e'}_{{\textbf{\textit l},\bf{x}}}(t,\cdot)$ the flow of the vector field on the right-hand side of (\ref{EqFExt}). If $e'$ is marked as incoming, then we define
\begin{equation} W^{u}(p_{e'})=\{q \in M: \lim_{t \rightarrow -\infty } \Phi^{e'}_{{\textbf{\textit l},\bf{x}}}(t,q)=p_{e'}\}, \label{EqDefUe} \end{equation} 
\begin{equation} \overline{W}^{u}(p_{e'})=\bigcup \limits_{p' \in Crit(f)} \overline{\mathcal{M}}(p_{e'},p') \times W^{u}(p') \label{EqDefUeBar} \end{equation}
and \begin{equation} M_{e'}= \{(q,q): q \in \overline{W}^{u}(p_{e'})\} \subset M \times M. \label{EqDefMe1} \end{equation}
Similarly, if $e'$ is marked as outgoing, then we consider
 \begin{equation} W^{s}(p_{e'})=\{q \in M: \lim_{t \rightarrow \infty } \Phi^{e'}_{{\textbf{\textit l},\bf{x}}}(t,q)=p_{e'}\}, \label{EqDefSe} \end{equation}
 \begin{equation} \overline{W}^{s}(p_{e'})=\bigcup \limits_{p' \in Crit(f)} W^{s}(p') \times \overline{\mathcal{M}}(p',p_{e'}) \label{EqDefSeBar} \end{equation} 
and \begin{equation} M_{e'}= \{(q,q): q \in \overline{W}^{s}(p_{e'}) \} \subset M \times M. \label{EqDefBarMe2} \end{equation}
It is well known that $\overline{W}^{u}(p_{e'})$ and $\overline{W}^{s}(p_{e'})$ are manifolds with corners, whose boundaries are given by respectively
\begin{equation}
\partial \overline{W}^{u}(p_{e'})=\coprod\limits_{ind_{f}(p_{e'})-ind_{f}(p)=1} \overline{\mathcal{M}}(p_{e'},p) \times \overline{W}^{u}(p) \label{EqBoundaryUe'} \end{equation} and
\begin{equation} \partial \overline{W}^{s}(p_{e'})=\coprod\limits_{ind_{f}(p)-ind_{f}(p_{e'})=1} \overline{W}^{s}(p) \times \overline{\mathcal{M}}(p,p_{e'}). \label{EqBoundarySe'} \end{equation} We refer to \cite{We} for a detailed study of trajectory spaces in Morse theory.\\ \\
We have thus far associated to every edge $e'$ of $\Gamma'_{0}$ a manifold with corners $M_{e'}$. As a submanifold of $M \times M$, $M_{e'}$ depends on the choice of ${\textbf{\textit l}} \in Met_{0}(\Gamma)$ and of the vector field datum ${\bf x}$, however the diffeomorphism type of $M_{e'}$ is independent of these choices. Let us 
define 
\begin{equation} N=Met_{0}(\Gamma) \times (M \times M)^{\times |E(\Gamma'_{0})|} \label{EqDefMfldN} \end{equation} and denote by $L_{\bf x}$ the subset \begin{equation} L_{\bf{x}} = \bigcup \limits_{{\textbf{\textit l}} \in Met_{0}(\Gamma)}(\{ \textbf{\textit l} \} \times (\times_{e' \in E(\Gamma')}M_{e'}({\textbf{\textit l}},{\bf{x}}))) \subset N. \label{EqDefL} \end{equation}  
As a product of manifolds with corners, $L_{\bf x}$ is again a manifold with corners. \\ \\
Next we consider the submanifold $L' \subset N$ defined as follows. Recall that each edge of $\Gamma'_{0}$ is oriented. We assign to each factor $M$ appearing on the right-hand side of (\ref{EqDefMfldN}) a vertex of $\Gamma'_{0}$ by the following rule: if $e' \in E(\Gamma'_{0})$ is incident to the vertices $v$ and $v'$ in this order, then to the two factors of the copy of $M \times M$ corresponding to $e'$ the pair $v$ and $v'$ respectively are associated. Write each element of $N$ as a $(\textbf{\textit l},{\bf q})$, where ${\bf q}$ is a tuple whose entries are points of $M$ and denote for each entry $q$ by $v(q)$ the vertex of $\Gamma'_{0}$ assigned to the corresponding factor of $M$. Then $L' \subset N$ is defined as the subset of all those tuples $({\textbf{\textit l}}, {\bf q})$, such that if for two entries $q$ and $q'$ of ${\bf q}$, $v(q)=v(q')$ is the same {\it internal} vertex of $\Gamma'_{0}$, then $q=q'$. Thus $L' \subset N$ is a fat diagonal which corresponds to the incidence relations of the internal vertices of the graph. \\ \\
We identify the space $\overline{\mathcal{M}}_{\Gamma,{\bf x}}({\bf p}_{+},{\bf p}_{-})$ as the intersection $L_{\bf{x}} \cap L'$: points of $L_{{\bf x}}$ are tuples of flow lines associated to the edges of the graph, while intersecting with $L'$ corresponds to imposing the incidence relations that stem from continuity at the internal vertices. It follows from Proposition \ref{PropGener} that the intersection $L_{\bf{x}} \cap L'$ is transverse. This establishes the first part of the Proposition. Formula (\ref{EqMGammaBoundary}) follows from (\ref{EqBoundaryUe'}) and (\ref{EqBoundarySe'}). The third claim is a consequence of the above discussion of compactifications of the spaces $\mathcal{M}_{\Gamma,{\bf x}}({\bf p}_{+},{\bf p}_{-})$. 
\end{proof}
Proposition \ref{PropMbarProperties} implies that the pair $(\mathcal{M}_{\Gamma,{\bf x}}({\bf p}_{+},{\bf p}_{-}),\pi_{\Gamma})$, together with a choice of orientation of $\mathcal{M}_{\Gamma,{\bf x}}({\bf p}_{+},{\bf p}_{-})$, defines an element of the chain complex $C^{BM}_{*}(\mathcal{M}_{\Sigma})$ introduced in Section \ref{SubSecConv}.
\section{The Operations} \label{SecCohOp}
This Section contains the discussion of orientations of the spaces of flow graphs as well as the proofs of Theorem \ref{ThCohOp} and of the gluing axiom.    
\subsection{Orientations of the Spaces of Flow Graphs} \label{SubSecOr}
The identification given in the proof of Proposition \ref{PropMbarProperties} of the space $\overline{\mathcal{M}}_{\Gamma,{\bf x}}({\bf p}_{+},{\bf p}_{-})$ as a transverse intersection of submanifolds can be used in the discussion of orientations: to orient $\overline{\mathcal{M}}_{\Gamma,{\bf x}}({\bf p}_{+},{\bf p}_{-})$, it suffices to fix orientations of the manifolds $L_{\bf{x}}$, $L'$ and $N$. It follws from the definition of these manifolds that their orientations can be determined by choosing orientations of the internal edges of $\Gamma$ as well as linear orderings of the vertices and of the edges of $\Gamma$. Moreover, it is straightforward to determine how the orientations of $L_{\bf{x}}$, $L'$ and $N$, and thus the orientation of $\overline{\mathcal{M}}_{\Gamma,{\bf x}}({\bf p}_{+},{\bf p}_{-})$, change when we reorder vertices and edges or reverse the orientation of an edge. The result of this discussion is summarized in the following Proposition.
\begin{proposition} \label{PropOrMGamma}
Suppose that for every critical point $p$ of $f$, an orientation of the unstable submanifold $W^{u}_{f}(p)$ is fixed.
\begin{enumerate}
\item An orientation of $\overline{\mathcal{M}}_{\Gamma,{\bf x}}({\bf p}_{+},{\bf p}_{-})$ is uniquely defined by choosing orientations of all the internal edges of $\Gamma$ and linear orderings of all the edges and of the internal vertices of $\Gamma$.
\item Reversing the orientation of an internal edge or interchanging two consecutive internal vertices changes the orientation of $\overline{\mathcal{M}}_{\Gamma,{\bf x}}({\bf p}_{+},{\bf p}_{-})$ by a $(-1)^{d}$, where $d$ is the dimension of $M$.  
\item Interchanging two consecutive edges $e_{i}$, $e_{j}$ changes the orientation of $\overline{\mathcal{M}}_{\Gamma,{\bf x}}({\bf p}_{+},{\bf p}_{-})$ by $(-1)^{k_{i}k_{j}}$, where the integers $k_{i}$ are given by the following rule. If $e_{i}$ is an internal edge of the graph, then $k_{i}=d+1$. If $e_{i}$ is an external edge with the correspoinding critical point $p_{i}$, then 
\begin{equation} k_{i}=\begin{cases}d-ind_{f}(p_{i})+1, & \text{ if } e_{i} \text{ is marked as incoming},\\ ind_{f}(p_{i})+1,& \text{ if }e_{i} \text { is marked as outgoing.} \end{cases} \end{equation}
\end{enumerate}
\end{proposition} 
We can now explain the local systems $det$ and $or$ on $\mathcal{M}_{\Sigma}$ which appear in Theorem \ref{ThCohOp}. A local trivialization of $det$ on $Met(\Gamma)/Aut(\Gamma) \subset \mathcal{M}_{\Sigma}$ is defined by a choice of orientations of the internal edges as well as of linear orderings of the vertices, of the internal edges and of those external edges of $\Gamma$, which are marked as incoming. Changing the orientation of an internal edge or interchanging two consecutive vertices or edges changes the sign of the trivialization. The fibre of $det$ can be identified with the determinant line of the cohomology $H^{*}(\Gamma,O_{-})$ of $\Gamma$ relative the vertices corresponding to the outgoing marked points. The local system $det$ is graded by assigning to each section the degree $\chi(\Sigma)-n_{-}$.\\ \\
The local system $or$ is the orientation sheaf of $\mathcal{M}_{\Sigma}$. Explicitly, a local trivialization of $or$ on $Met(\Gamma)/Aut(\Gamma) \subset \mathcal{M}_{\Sigma}$ is defined by a choice of orientations of the internal edges and of linear orderings of all the vertices and edges of $\Gamma$. As before, reversing the orientation of an edge or interchanging consecutive vertices or edges reverses the sign of the trivializazion. The fibre of the local system $or$ can be identified with the determinant line of $H^{*}(\Gamma)$. The grading of the system $or$ is trivial, i. e. every section has degree zero.
\begin{lemma} \label{LemmaOrMGamma}
Suppose that every vertex of the ribbon graph $\Gamma$ has odd valency. Then the pair $(\overline{\mathcal{M}}_{\Gamma,{\bf x}}({\bf p}_{+},{\bf p}_{-}),\pi_{\Gamma})$ defines an element of $C^{BM}_{*}(\mathcal{M}_{\Sigma};det^{\otimes d} \otimes or)$.
\end{lemma}
\begin{proof}
We first note that since every external edge of $\Gamma$ is marked as either incoming or outgoing and thus has a natural orientation, we could equivalently define the local system $det$ by considering linear orderings of {\it all} the vertices instead of only the internal ones. Indeed, if $e$ is an external edge with the incident vertices $v$ and $v'$, where $v$ is an internal and $v'$ an external vertex, then we insert $v'$ into a given ordering of the internal vertices as either the predecessor or the successor of $v$, according to the orientation of $e$.\\ \\   
We must show that an orientation of $\overline{\mathcal{M}}_{\Gamma,{\bf x}}({\bf p}_{+},{\bf p}_{-})$ is the same as a trivialization of the pullback under $\pi_{\Gamma}$ of the local system $det^{\otimes d} \otimes or$. Comparing the definitions of $det$ and $or$ with the result of Proposition \ref{PropOrMGamma}, it suffices to check that a trivialization of $or$ is given by a choice an orientation of the vector space spanned by the edges of the graph. This follows from the observation going back to J. Conant and K. Vogtmann (\cite{CoVo}, Corollary 1) that if all vertices of $\Gamma$ have odd valency, then there is a natural orientation of the vector space spanned by the vertices and the half-edges. 
\end{proof}      
\subsection{Proof of Theorem \ref{ThCohOp}} \label{SubSecProofCohOp}
We can now complete the proof of Theorem \ref{ThCohOp}. Recall that the top-dimensional cells in the ribbon graph decomposition (\ref{EqRibbonGrDec}) are labelled by the ribbon graphs whose internal vertices have valency three. By Proposition 
\ref{PropMbarProperties} and Lemma \ref{LemmaOrMGamma}, to each such graph $\Gamma$ and each choice of a vector field datum ${\bf x} \in \mathcal{X}_{\Gamma, reg}$ is associated a geometric chain \begin{equation} Z_{\Gamma,{\bf x}}=(\overline{\mathcal{M}}_{\Gamma,{\bf x}}({\bf p}_{+},{\bf p}_{-}),\pi_{\Gamma}) \in C^{BM}_{*}(\mathcal{M}_{\Sigma};det^{\otimes d} \otimes or).  \label{EqChainZGamma} \end{equation} Moreover, we may assume that the vector field data for different graphs $\Gamma$ are chosen as in the second part of Corollary \ref{CorTransv}.\\ \\
To prove that 
$$ F^{M}_{\Sigma}\colon  (C^{*}(f))^{\otimes n_{+}}  \rightarrow C^{BM}_{*}(\mathcal{M}_{\Sigma};det^{\otimes d} \otimes or) \otimes (C^{*}(f))^{\otimes n_{-}}, $$ $${\bf p}_{+} \mapsto \sum \limits_{\Gamma, {\bf p}_{-}} Z_{\Gamma,{\bf x}}({\bf p}_{+}, {\bf p}_{-}) \otimes {\bf p}_{-}$$
is a cochain map, we compute $(\sum_{\Gamma, {\bf p}_{-}} \partial Z_{\Gamma,{\bf x}}({\bf p}_{+}, {\bf p}_{-}) \otimes {\bf p}_{-})$. By the second part of Proposition \ref{PropMbarProperties}, the latter expression is a sum of terms of two types: the first type corresponds to breaking of trajectories at external edges (see the expressions in the first two lines of (\ref{EqMGammaBoundary})), while terms of the second type correspond to collapsing an internal edge of $\Gamma$ (see the expression in the third line of (\ref{EqMGammaBoundary})). The summands of the first type yield \begin{equation} \begin{split}  \sum \limits_{\Gamma,{\bf p}_{-}} Z_{\Gamma,{\bf x}}(d{\bf p}_{+}, {\bf p}_{-}) \otimes {\bf p}_{-} &- \sum \limits_{\Gamma,{\bf p}_{-}} (-1)^{dim \text{ }\overline{\mathcal{M}}_{\Gamma,{\bf x}}({\bf p}_{+},{\bf p}_{-})} Z_{\Gamma,{\bf x}}({\bf p}_{+}, {\bf p}_{-})\otimes d {\bf p}_{-}. \end{split}\label{EqCompDel1} \end{equation}
We must show that the sum of all the terms of the second type is zero. These are of the form $\pm (\overline{\mathcal{M}}_{\widetilde{\Gamma},{\bf x}}({\bf p}_{+},{\bf p}_{-}),\pi_{\widetilde{\Gamma}})$, where $\widetilde{\Gamma}$ is obtained by collapsing a single internal edge in a ribbon graph $\Gamma$, all of whose internal vertices have valency three. The sign is determined as follows: $\overline{\mathcal{M}}_{\widetilde{\Gamma},{\bf x}}({\bf p}_{+},{\bf p}_{-})$ is oriented as a boundary component of $\overline{\mathcal{M}}_{\Gamma,{\bf x}}({\bf p}_{+},{\bf p}_{-})$ and the trivialization of the pullback to $\overline{\mathcal{M}}_{\widetilde{\Gamma},{\bf x}}({\bf p}_{+},{\bf p}_{-})$ of the local system $det^{\otimes d} \otimes or$ is induced by the trivialization of the pullback to $\overline{\mathcal{M}}_{{\Gamma},{\bf x}}({\bf p}_{+},{\bf p}_{-})$.\\ \\ We observe that in the sum of the terms of second type exactly those ribbon graphs $\widetilde{\Gamma}$ appear, where there is a unique internal vertex of valency four and all the remaining internal vertices have valency three. For each such $\widetilde{\Gamma}$, there are exactly two distinct pairs $(\Gamma_{1},e_{1})$ and $(\Gamma_{2},e_{2})$, so that $\widetilde{\Gamma}$ is obtained from $\Gamma_{1}$ and $\Gamma_{2}$ by collapsing the internal edges $e_{1} \in E(\Gamma_{1})$ and $e_{2} \in E(\Gamma_{2})$ respectively: $\Gamma_{1}$ and $\Gamma_{2}$ arise from the two different ways of expanding the four-valent vertex of $\widetilde{\Gamma}$ into two trivalent vertices.  
\begin{figure}[!htbp]
\centering	\floatbox[{\capbeside\thisfloatsetup{capbesideposition={left,center},capbesidewidth=3cm}}]{figure}[\FBwidth]{\caption{The two different ways of expanding a single four-valent vertex of a ribbon graph into two trivalent vertices.}	\label{Fig5}} {\includegraphics[width=0.6\textwidth]{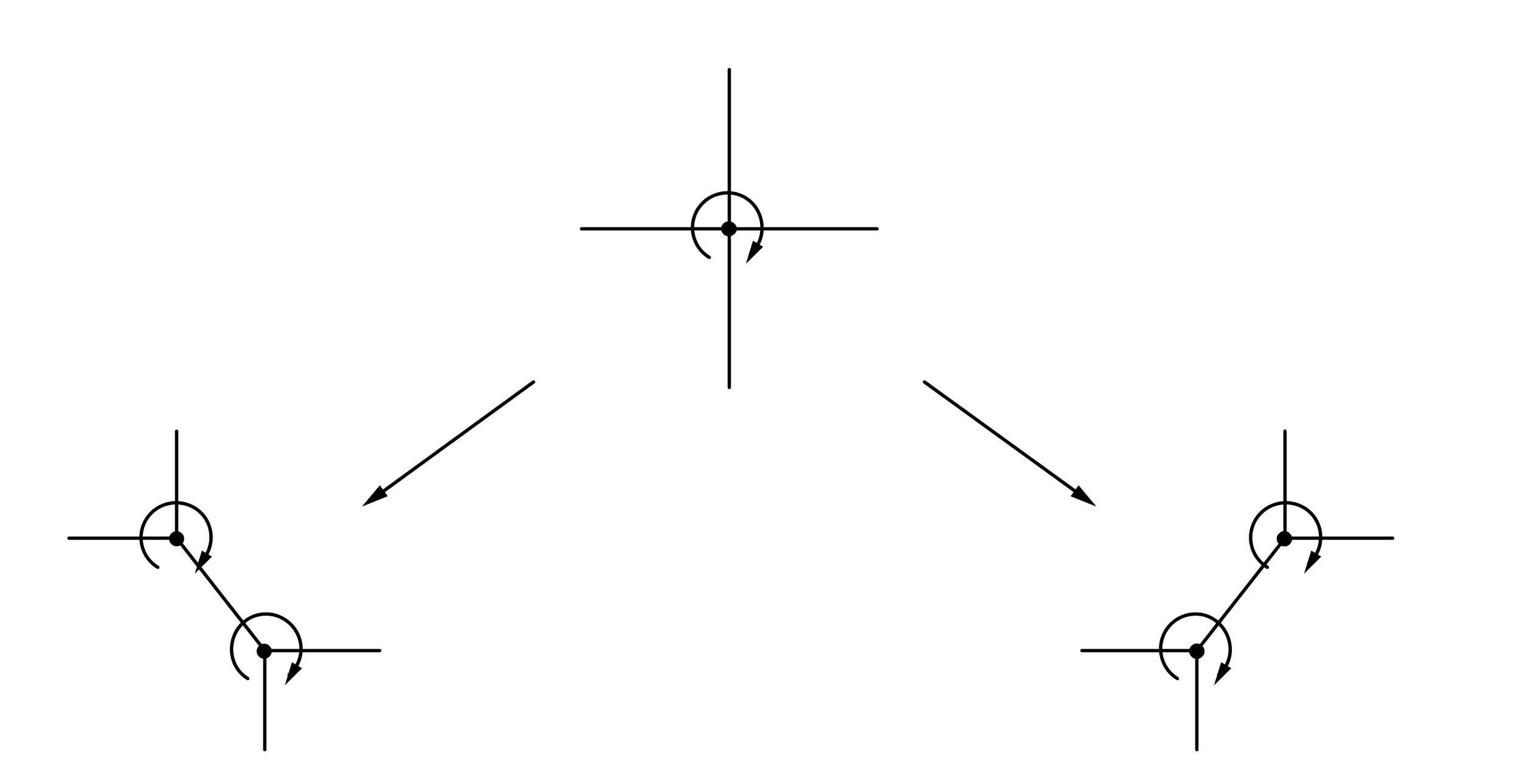}}
\end{figure} \\ \\
To complete the proof of the first part of the Theorem, it suffices to show:
\begin{lemma} \label{LemmaProofMain1.1}
The two copies of $(\overline{\mathcal{M}}_{\widetilde{\Gamma},{\bf x}}({\bf p}_{+},{\bf p}_{-}),\pi_{\widetilde{\Gamma}})$ corresponding to boundary components of $(\Gamma_{1},e_{1})$ and of $(\Gamma_{2},e_{2})$ enter the sum with the opposite sign.
\end{lemma}
\begin{proof} 
Assume first that $d$ is even. In this case by Proposition \ref{PropOrMGamma}, an orientation of $\overline{\mathcal{M}}_{{\Gamma}_{1},{\bf x}}({\bf p}_{+},{\bf p}_{-})$ is the same as an orientation of the vector space $W_{E(\Gamma_{1})}$ spanned by the edges of $\Gamma_{1}$. The boundary orientation of $\overline{\mathcal{M}}_{\widetilde{{\Gamma}},{\bf x}}({\bf p}_{+},{\bf p}_{-}) \subset \partial \overline{\mathcal{M}}_{{\Gamma_{1}},{\bf x}}({\bf p}_{+},{\bf p}_{-})$ is determined by requiring that the projection
\begin{equation} W_{E(\Gamma_{1})} \simeq \R_{e_{1}} \oplus W_{E(\widetilde{\Gamma})} \rightarrow W_{E(\widetilde{\Gamma})} \label{EqOrW1} \end{equation} to the second factor
be orientation-preserving. Here the symbol $\R_{e_{1}}$ denotes the direct summand of $\R$ corresponding to the edge $e_{1}$. The orientation of $\overline{\mathcal{M}}_{\widetilde{{\Gamma}},{\bf x}}({\bf p}_{+},{\bf p}_{-})$ as a boundary component of $ \overline{\mathcal{M}}_{{\Gamma_{2}},{\bf x}}({\bf p}_{+},{\bf p}_{-})$ is defined analogously. \\ \\
It follows from the definition that a trivialization of the pullback of the local system $or$ to $\overline{\mathcal{M}}_{{\Gamma}_{1},{\bf x}}({\bf p}_{+},{\bf p}_{-})$ is given by an orientation of the vector space
$W_{E(\Gamma_{1})} \oplus W_{V(\Gamma_{1})} \oplus W_{H(\Gamma_{1})} $, where $W_{V(\Gamma_{1})}$ and $W_{H(\Gamma_{1})}$ denote the vector spaces generated by the vertices and by the half-edges of $\Gamma_{1}$ respectively. The projection $W_{E(\Gamma_{1})} \rightarrow W_{E(\widetilde{\Gamma})}$ was described above, while the projection $W_{V(\Gamma_{1})} \oplus W_{H(\Gamma_{1})}  \rightarrow W_{V(\widetilde{\Gamma})} \oplus W_{H(\widetilde{\Gamma})} $ is given as follows. Denote by $h_{1}$ and $h'_{1}$ the half-edges of $e_{1}$, by $v_{1}$ and $v'_{1}$ the corresponding vertices of $\Gamma_{1}$ and by $v \in V(\widetilde{\Gamma})$ the vertex to which $e_{1}$ is collapsed. We identify \begin{equation} W_{V(\Gamma_{1})-\{v',v'_{1}\}} \simeq W_{V(\widetilde{\Gamma})-\{v\}} \label{EqNotOr1} \end{equation} as well as  \begin{equation} W_{H(\Gamma_{1})-\{h_{1},h'_{1}\}} \simeq W_{H(\widetilde{\Gamma})} \label{EqNotOr2} \end{equation} and consider the map
 $$ W_{V(\Gamma_{1})} \oplus W_{H(\Gamma_{1})}   \simeq (\R_{h_{1}} \oplus \R_{h'_{1}} \oplus \R_{v_{1}})  \oplus (\R_{v'_{1}} \oplus W_{V(\Gamma_{1})-\{v_{1},v'_{1}\}} \oplus W_{H(\Gamma_{1})-\{h_{1},h'_{1}\}})$$ \begin{equation} \, \, \, \rightarrow \R_{v}  \oplus W_{V(\widetilde{\Gamma})-\{v\}} \oplus W_{H(\widetilde{\Gamma})} \simeq W_{V(\widetilde{\Gamma})} \oplus W_{H(\widetilde{\Gamma})},   \label{EqOrW2}
\end{equation}
where the arrow denotes projection to the second factor. A trivialization of $or$ over $\overline{\mathcal{M}}_{{\Gamma_{1}},{\bf x}}({\bf p}_{+},{\bf p}_{-})$ induces a trivialization over $\overline{\mathcal{M}}_{{\widetilde{\Gamma}},{\bf x}}({\bf p}_{+},{\bf p}_{-})$ by requiring that the projection \begin{equation} W_{E(\Gamma_{1})} \oplus ( W_{V(\Gamma_{1})} \oplus W_{H(\Gamma_{1})}) \rightarrow W_{E(\widetilde{\Gamma})} \oplus ( W_{V(\widetilde{\Gamma})} \oplus W_{H(\widetilde{\Gamma})}) \label{EqOrW3} \end{equation} whose components are given by (\ref{EqOrW1}) and (\ref{EqOrW2}) be orientation-preserving. In the same way a trivialization over $\overline{\mathcal{M}}_{\widetilde{\Gamma},{\bf x}}({\bf p}_{+},{\bf p}_{-})$ is induced by a trivialization over $\overline{\mathcal{M}}_{{\Gamma}_{2},{\bf x}}({\bf p}_{+},{\bf p}_{-})$.\\ \\
Recall from the proof of Lemma \ref{LemmaProofMain1.1} that since $\Gamma_{1}$ and $\Gamma_{2}$ have vertices of odd valency, the ribbon structures of $\Gamma_{1}$ and of $\Gamma_{2}$ define orientations of the vector spaces $W_{V(\Gamma_{1})} \oplus W_{H(\Gamma_{1})}$ and $W_{V(\Gamma_{2})} \oplus W_{H(\Gamma_{2})} $ respectively. It is immidiate to check using the explicit description given by (\ref{EqOrW2}) that the two orientations on $W_{V(\widetilde{\Gamma})} \oplus W_{H(\widetilde{\Gamma})}$ induced by the projections $$ W_{V(\Gamma_{1})} \oplus W_{H(\Gamma_{1})} \rightarrow W_{V(\widetilde{\Gamma})} \oplus W_{H(\widetilde{\Gamma})} $$ and $$W_{V(\Gamma_{2})} \oplus W_{H(\Gamma_{2})} \rightarrow  W_{V(\widetilde{\Gamma})} \oplus W_{H(\widetilde{\Gamma})}$$ are opposite. Thus if the orientations of $\overline{\mathcal{M}}_{{\widetilde{\Gamma}},{\bf x}}({\bf p}_{+},{\bf p}_{-})$ as boundary of $\overline{\mathcal{M}}_{{\Gamma}_{1},{\bf x}}({\bf p}_{+},{\bf p}_{-})$ and as boundary of $\overline{\mathcal{M}}_{{\Gamma_{2}},{\bf x}}({\bf p}_{+},{\bf p}_{-})$ coincide, i. e. the projection (\ref{EqOrW1}) and the corresponding projection for $\Gamma_{2}$ are both orientation-preserving, then the orientations of $W_{E(\widetilde{\Gamma})} \oplus W_{V(\widetilde{\Gamma})} \oplus W_{H(\widetilde{\Gamma})}$ induced by the projection    
(\ref{EqOrW3}) and by the corresponding projection for $\Gamma_{2}$ are opposite. In this case the two trivializations of $or$ over $\overline{\mathcal{M}}_{\widetilde{\Gamma},{\bf x}}({\bf p}_{+},{\bf p}_{-})$ differ by a sign. This completes the proof of the Lemma in the case when $d$ is even. The case of odd $d$ follows by the same argument, but with $W_{E(\Gamma)}$ replaced everywhere by $W_{det} \oplus W_{E(\Gamma)}$, where $W_{det}$ is the vector space whose orientation corresponds to a trivialization of $det$ (explicitly, $W_{det}=W_{E(\Gamma)-E_{-}(\Gamma)} \oplus W_{V(\Gamma)} \oplus W_{H(\Gamma)}$, where $E_{-}(\Gamma)$ denote the outgoing external edges of $\Gamma$). 
\end{proof}
We now turn to the proof of the second part of the Theorem. By the first part, Proposition \ref{PropIdBMHom} and the isomorphism (\ref{EqPDBM}), $F^{M}_{\Sigma}$ induces a linear map
\begin{equation} H F^{M}_{\Sigma}: (H^{*}(M))^{\otimes n_{+}} \rightarrow H^{*}(\mathcal{M}_{\Sigma}; det^{d}) \otimes (H^{*}(M))^{\otimes n_{-}}. \label{EqHfSigmaProof} \end{equation}
 Comparing the definition of the grading of the local system $det$ with the index formula of Proposition \ref{PropGener}, one finds that the degree of $HF_{\Sigma}$ is zero (note that if every internal vertex of $\Gamma$ has valency three, then the number $|E(\Gamma)|$ of edges of $\Gamma$ coincides with the dimension of $\mathcal{M}_{\Sigma}$).\\ \\ The proof of the independence of the operations $HF^{M}_{\Sigma}$ of the choices relies on a flow graph version of the continuation argument, which is classically used to prove the invariance of Morse homology. Given two triples $(f,g,{\bf x})$ and $(f',g',{\bf x}')$, we fix a smooth one-paramater family $(f_{t},g_{t},{\bf x}_{t})_{t \in \R}$ so that
\begin{equation}
(f_{t},g_{t},{\bf x}_{t})= \begin{cases} (f,g,{\bf x}) & \text{ if } t \leq -1 \\
(f',g',{\bf x}') & \text{ if } t \geq 1. \end{cases}
\end{equation} Recall that the classical continuation principle consists of the following: one observes that for suitable choice of the one-parameter family $(f_{t},g_{t})_{t \in \R}$, each of the spaces
\begin{equation} \mathcal{N}(p,p')= \{\gamma: \R \rightarrow M, \dot{\gamma}(t)=\nabla_{g_{t}}f_{t}(\gamma(t)), \lim_{t \rightarrow -\infty } \gamma(t)=p, \lim_{t \rightarrow \infty } \gamma(t)=p' \}, \label{EqDefNpp} \end{equation}
where $p \in Crit(f)$, $p' \in Crit(f')$, is a compact oriented manifold. One then shows that the count of the elements of the zero-dimensional spaces yields a quasi-isomorphism
$$\Psi: C^{*}(f,g) \rightarrow C^{*}(f',g'), $$ 
\begin{equation}\Psi: p \mapsto \sum \limits_{ind_{f}(p')=ind_{f}(p)}|\mathcal{N}(p,p')|p'. \label{EqConstrPsi} \end{equation}
Now denote by $$ \Phi^{M}_{\Sigma}\colon  (C^{*}(f',g'))^{\otimes n_{+}}  \rightarrow C^{BM}_{*}(\mathcal{M}_{\Sigma};det^{d} \otimes or) \otimes (C^{*}(f',g'))^{\otimes n_{-}} $$ the map associated by the construction of the first part of the Theorem to the triple $(f',g',{\bf x}')$. We will construct a chain homotopy $\Theta$ between $\Phi^{M}_{\Sigma} \circ \Psi^{\otimes n_{+}}$ and $(Id \otimes \Psi^{n_{-}}) \circ F^{M}_{\Sigma}$. This chain homotopy will be obtained by studying spaces $\overline{\mathcal{N}}_{\Gamma}({\bf p}_{+},{\bf p}'_{-})$ which we now introduce. \\ \\
Recall from Definition \ref{Def2} that $\mathcal{M}_{{\Gamma},{\bf x}}({\bf p}_{+},{\bf p}_{-})$ is the zero locus of ${\bf S}-\mathcal{F}_{\mathbf{x}}$, where ${\bf S}$ and $\mathcal{F}_{\mathbf{x}}$ are sections of a Banach bundle $\mathcal{E}$ over $Met(\Gamma) \times \mathcal{B}_{\Gamma}({\bf p}_{+},{\bf p}_{-})$ (see Definition \ref{DefB_Gamma}). Given ${\bf p}_{+} \in Crit(f)^{\times n_{+}}$ and ${\bf p}'_{-} \in Crit(f')^{\times n_{-}}$, we define $\mathcal{N}_{\Gamma}({\bf p}_{+},{\bf p}'_{-}) \subset Met(\Gamma) \times \mathcal{B}_{\Gamma}({\bf p}_{+},{\bf p}'_{-})$ as the union $$\mathcal{N}_{\Gamma}({\bf p}_{+},{\bf p}'_{-})=\cup_{T \in \R} \mathcal{N}_{\Gamma,T}({\bf p}_{+},{\bf p}'_{-}),$$ where $\mathcal{N}_{\Gamma,T}({\bf p}_{+},{\bf p}'_{-})$ is the zero locus of the section ${\bf S}-\mathcal{F}_{T}$ of $\mathcal{E}$, with ${\bf S}$ is defined as in (\ref{EqDefSectionS}) and $\mathcal{F}_{T}$ given as follows. If $e'$ is an external edge of $\Gamma'$, then \begin{equation} F_{e',T}({\textbf{\textit l}},{\bf \gamma})(t)=\nabla_{g_{t+T}}f_{t+T}(\gamma_{e'}(t))+\sigma(t)x_{e',t+T}(\textbf{\textit l},|t|,\gamma_{e'}(t)). \label{EqFExt2} \end{equation} If $e'$ is an internal edge of $\Gamma'$ which is mapped to $e \in E(\Gamma)$ under the projection $\Gamma' \rightarrow \Gamma$, then \begin{equation} F_{e',T}({\textbf{\textit l}},{\bf \gamma})(t)=l_{e}x_{e',T}(\textbf{\textit l},t,\gamma_{e'}(t)) \label{EqFInt2} \end{equation}  
(compare with (\ref{EqFExt}) and (\ref{EqFInt}) respectively). We denote by $\overline{\mathcal{N}}_{\Gamma}({\bf p}_{+},{\bf p}'_{-})$ the partial compactification of $\mathcal{N}_{\Gamma}({\bf p}_{+},{\bf p}'_{-})$ given by
\begin{equation} \begin{split} \overline{\mathcal{N}}_{\Gamma}({\bf p}_{+},{\bf p}'_{-})= & \quad \, 
\bigcup  \limits_{{\bf q}_{+}, {\bf q}'_{-}, \widetilde{\Gamma} \prec \Gamma}  \overline{\mathcal{M}}({\bf p}_{+},{\bf q}_{+}) \times \mathcal{N}_{\widetilde{\Gamma},{\bf x}}({\bf q}_{+},{\bf q}'_{-}) \times \overline{\mathcal{M}}'({\bf q}'_{-},{\bf p}'_{-}) \\& 
\bigcup (\bigcup \limits_{|{\bf q}'_{+}|=|{\bf p}_{+}|} \mathcal{N}({\bf p}_{+},{\bf q}'_{+}) \times \overline{\mathcal{M}}_{\Gamma,{\bf x'}}({\bf q}'_{+},{\bf p}'_{-}))\\
& \bigcup (\bigcup \limits_{|{\bf q}_{-}|=|{\bf q}'_{-}|} \overline{\mathcal{M}}_{\Gamma,{\bf x}}({\bf p}_{+},{\bf q}_{-})) \times \mathcal{N}({\bf q}_{-},{\bf p}'_{-})), 
\end{split} \label{EqDefNGammaBar} \end{equation}
where $\mathcal{N}({\bf p},{\bf p}')$ denotes the product
\begin{equation}\mathcal{N}({\bf p},{\bf p}')=\mathcal{N}(p_{1},p'_{1}) \times \dots \times \mathcal{N}(p_{n},p'_{n}) \label{EqDefNWOGamma} \end{equation}
(note that if $|{\bf p}|=|{\bf p}'|$, then the space $\mathcal{N}({\bf p},{\bf p}')$ is empty unless $ind_{f}({\bf p}_{j})=ind_{f'}({\bf p}'_{j})$ for $j=1, \dots, n$.)\\ \\
The meaning of the terms of the right-hand side of (\ref{EqDefNGammaBar}) is as follows. The expression in the first line corresponds to allowing internal edges of length zero as well as breaking of trajectories along the external edges of $\Gamma$ for fixed $T \in \R$. Thus if we denote by $\overline{\mathcal{N}}_{\Gamma,T}({\bf p}_{+},{\bf p}'_{-})$ the partial compactification of $\mathcal{N}_{\Gamma,T}({\bf p}_{+},{\bf p}'_{-})$ defined as in \ref{DefMGammaxBar}, then the union in the first line of (\ref{EqDefNGammaBar}) can be identified with $\cup_{T \in \R} \overline{\mathcal{N}}_{\Gamma,T}({\bf p}_{+},{\bf p}'_{-})$. The terms in the second and third line correspond to partial compactifications at
$T \rightarrow \infty$ and $T \rightarrow - \infty$ respectively.\\ \\
By a slight abuse of notation, we will again denote by $\pi_{\Gamma}: \overline{\mathcal{N}}_{\Gamma}({\bf p}_{+},{\bf p}'_{-}) \rightarrow \mathcal{M}_{\Sigma}$ the natural projection.
\begin{lemma} \label{LemmaChainHomConstr}
\begin{enumerate}
\item There exists a one-parameter family $({\bf x}_{t})_{t \in \R}$, so that every pair
\begin{equation} Y_{\Gamma}({\bf p}_{+},{\bf p}'_{-}) =(\overline{\mathcal{N}}_{\Gamma}({\bf p}_{+},{\bf p}'_{-}),\pi_{\Gamma}) \label{EqDefYGamma} \end{equation} defines an element of
$C^{BM}_{*}(\mathcal{M}_{\Sigma},det^{\otimes d} \otimes or)$.
\item Define \begin{equation} \begin{split} \Theta: (C^{*}(f))^{\otimes n_{+}}  & \rightarrow C^{BM}_{*}(\mathcal{M}_{\Sigma};det^{\otimes d} \otimes or) \otimes (C^{*}(f',g'))^{\otimes n_{-}}, \\ {\bf p}_{+} & \mapsto \sum \limits_{\Gamma,{\bf p}'_{-}}Y_{\Gamma}({\bf p}_{+},{\bf p}'_{-}) \otimes {\bf p}'_{-}. \end{split}  \end{equation}
\end{enumerate}
Then \begin{equation} d \circ \Theta+\Theta \circ d =
(Id \otimes \Psi^{\otimes n_{-}}) \circ F^{M}_{\Sigma}-\Phi^{M}_{\Sigma} \circ (\Psi^{\otimes n_{+}}) \label{EqChainHomClaim} \end{equation}
\end{lemma} 
This Lemma completes the proof of the second part of Theorem \ref{ThCohOp}: it follows from (\ref{EqChainHomClaim}) that the maps $HF^{M}_{\Sigma}$ and $H\Phi^{M}_{\Sigma}$ defined by $(f,g,{\bf x})$ and $(f',g',{\bf x}')$ respectively coincide up to the isomorphism $ H^{*}(M) \rightarrow H^{*}(M)$ induced by $\Psi$. 
\begin{proof}
The first part of the Lemma follows by analogous arguments as in the proofs of Propositions \ref{PropGener} and \ref{PropMbarProperties}. From (\ref{EqDefNGammaBar}), the boundary of $\overline{\mathcal{N}}_{\Gamma}({\bf p}_{+},{\bf p}'_{-})$ may be identified as the disjoint union
\begin{equation} \begin{split} \partial \overline{\mathcal{N}}_{\Gamma}({\bf p}_{+},{\bf p}'_{-})= & \quad \, \, \, \, \, \, \, \,  
\coprod \limits_{|{\bf q}_{+}|- |{\bf p}_{+}|=1}  \mathcal{M}({\bf p}_{+},{\bf q}_{+}) \times \overline{\mathcal{N}}_{\Gamma,{\bf x}}({\bf q}_{+},{\bf p}'_{-})  \\&
\coprod (\, \, \, \, \coprod  \limits_{|{\bf p}'_{-}|- |{\bf q}'_{-}|=1} \overline{\mathcal{N}}_{\Gamma,{\bf x}}({\bf q}_{+},{\bf p}'_{-}) \times \mathcal{M}'({\bf q}'_{-},{\bf p}'_{-}))\\& 
\coprod( \, \quad \coprod \limits_{|{\bf q}'_{+}|=|{\bf p}_{+}|} \mathcal{N}({\bf p}_{+},{\bf q}'_{+}) \times \overline{\mathcal{M}}_{\Gamma,{\bf x'}}({\bf q}'_{+},{\bf p}'_{-}))\\
& \coprod (\, \quad \coprod \limits_{|{\bf q}_{-}|=|{\bf q}'_{-}|} \overline{\mathcal{M}}_{\Gamma,{\bf x}}({\bf p}_{+},{\bf q}_{-}) \times \mathcal{N}({\bf q}_{-},{\bf p}'_{-})) \\& \coprod (\coprod \limits_{e \in E_{int}(\Gamma) - L(\Gamma)} \overline{\mathcal{N}}_{\Gamma/e}({\bf p}_{+},{\bf p}'_{-})).
\end{split} \label{EqDelNGammaBar} \end{equation}
The sum over all $\Gamma$ of the geometric chains corresponsing to the terms on the left-hand side and in the first two lines of the right-hand side of (\ref{EqDelNGammaBar}) yields the left-hand side of (\ref{EqChainHomClaim}), while the expressions in the third and fourth lines correspond to the right-hand side of (\ref{EqChainHomClaim}). Finally, using Lemma \ref{LemmaProofMain1.1}, the sum of the chains corresponding to the terms in the last line of (\ref{EqDelNGammaBar}) vanishes. 
\end{proof}
\subsection{The Gluing Axiom} \label{SubSecGlue}
The goal of this Section is to show that the operations constructed in Theorem \ref{ThCohOp} are compatible with the gluing of surfaces and thus form what is called a homological conformal field theory. \\ \\
Let $\Sigma_{1}$ and $\Sigma_{2}$ be two surfaces as in Theorem \ref{ThCohOp} and assume that the number $n_{1-}$ of the outgoing marked points of $ \Sigma_{1}$ coincides with the number $n_{2+}$ of the incoming marked points of $\Sigma_{2}$. Denote this common number by $n$ and denote by $\Sigma$ the compact orieted surface obtained by attaching $\Sigma_{1}$ to $\Sigma_{2}$ along closed disjoint intervals around the outgoing marked points of $\Sigma_{1}$ respectively around the incoming marked points of $\Sigma_{2}$. Using the homeomorphism (\ref{EqRibbonGrDec}) of the ribbon graph decomposition, there is a map
\begin{equation}\Xi: \mathcal{M}_{\Sigma_{1}} \times \mathcal{M}_{\Sigma_{2}} \rightarrow \mathcal{M}_{\Sigma} \label{EqGluinOnMSigma} \end{equation} 
defined by first attaching the outgoing edges of a metric ribbon graph corresponding to $\Sigma_{1}$ to the incoming edges of a metric ribbon graph corresponding to $\Sigma_{2}$ and then erasing the bivalent vertices from the resulting graph. Denote by $\pi_{1}: \mathcal{M}_{\Sigma_{1}} \times \mathcal{M}_{\Sigma_{2}} \rightarrow \mathcal{M}_{\Sigma_{1}}$ and by $\pi_{2}: \mathcal{M}_{\Sigma_{1}} \times \mathcal{M}_{\Sigma_{2}} \rightarrow \mathcal{M}_{\Sigma_{2}}$ the projection to the first respectively to the second factor. The local system $det$ is compatible with gluing, i. e. $\Xi^{*} det \simeq \pi_{1}^{*}det \oplus \pi_{2}^{*}det$. Thus $\Xi$ induces a map
\begin{equation} \Xi^{*}: H^{*}(\mathcal{M}_{\Sigma};det^{\otimes d}) \rightarrow H^{*}(\mathcal{M}_{\Sigma_{1}};det^{\otimes d}) \otimes H^{*}(\mathcal{M}_{\Sigma_{2}};det^{\otimes d}). \label{DefXiHom} \end{equation} 
The next Proposition establishes the compatibility of the operations constructed in Theorem \ref{ThCohOp} with the gluing maps (\ref{DefXiHom}).
\begin{proposition} \label{PropGlueAxiom}
The following diagram commutes:\\ \\ 
\vspace{0.3cm}
\begin{xy} \xymatrix @C=0.45in{ 
(H^{*}(M))^{\otimes n_{+}} \ar[r]^/-2em/{HF^{M}_{\Sigma}} \ar[d]^/0em/{HF^{M}_{\Sigma_{1}}} & H^{*}(\mathcal{M}_{\Sigma};det^{\otimes d}) \otimes (H^{*}(M))^{\otimes n_{-}} \ar[d]^/0em/{\Xi^{*} \otimes Id}\\ 
  H^{*}(\mathcal{M}_{\Sigma_{1}};det^{\otimes d}) \otimes (H^{*}(M))^{\otimes n} \ar[r]^/-2em/{Id \otimes HF^{M}_{\Sigma_{2}}} & H^{*}(\mathcal{M}_{\Sigma_{1}};det^{\otimes d}) \otimes H^{*}(\mathcal{M}_{\Sigma_{2}};det^{\otimes d}) \otimes  (H^{*}(M))^{\otimes n_{-}}.   
}\end{xy}\\ 
\end{proposition} 
Before giving the proof of this Proposition, we gather some preliminary arguments. First, let us express the gluing homomorphism (\ref{DefXiHom}) using geometric homology (we will for simplicity leave out the local coefficient systems in the notation). Via the Poincar\'e duality isomorphisms (\ref{EqPDBM}) on $\mathcal{M}_{\Sigma}$, $\mathcal{M}_{\Sigma_{1}}$ and $\mathcal{M}_{\Sigma_{2}}$, one obtains from $\Xi^{*}$ the corresponding transfer homomorphism
\begin{equation}\Xi^{!BM}_{*}: H^{BM}_{*}(\mathcal{M}_{\Sigma}) \rightarrow H^{BM}_{*+n}(\mathcal{M}_{\Sigma_{1}} \times \mathcal{M}_{\Sigma_{2}}). \label{EqDefMapHBMXi} \end{equation}
The remarks made in section \ref{SubSecConv} allow to identify $\Xi^{!BM}_{*}$ as the homomorphism induced by an explicit chain map \begin{equation} \Xi^{!CBM}_{*}: C^{BM}_{*}(\mathcal{M}_{\Sigma}) \rightarrow C^{BM}_{*+n}(\mathcal{M}_{\Sigma_{1}} \times \mathcal{M}_{\Sigma_{2}}).  \label{EqDefMapCBMXi} \end{equation}
To this end, we observe that the image $$\Xi(\mathcal{M}_{\Sigma_{1}} \times \mathcal{M}_{\Sigma_{2}}) \subset \mathcal{M}_{\Sigma}$$ is an open subset: it consists of the equivalence classes of all the metric ribbon graphs $\Gamma$ which can be obtained by gluing together graphs $\Gamma_{1}$ and $\Gamma_{2}$ from the ribbon graph decompositions of $\mathcal{M}_{\Sigma_{1}}$ and $\mathcal{M}_{\Sigma_{2}}$ respectively. There is a homeomorphism \begin{equation} \mathcal{M}_{\Sigma_{1}} \times \mathcal{M}_{\Sigma_{2}} \simeq \Xi(\mathcal{M}_{\Sigma_{1}} \times \mathcal{M}_{\Sigma_{2}}) \times \R^{n}_{+} \label{EqHomeoMSigmaSigma'} \end{equation} so that $\Xi$ can be identified as the composition of the projection $\pi$ to the first factor of $\Xi(\mathcal{M}_{\Sigma_{1}} \times \mathcal{M}_{\Sigma_{2}}) \times \R^{n}_{+}$ with the inclusion $i: \Xi(\mathcal{M}_{\Sigma_{1}} \times \mathcal{M}_{\Sigma_{2}}) \hookrightarrow \mathcal{M}_{\Sigma}$. This yields the following map (\ref{EqDefMapCBMXi}): given a genereator $(P,f_{P})$ of $C^{BM}_{*}(\mathcal{M}_{\Sigma})$, we first intersect with $\Xi(\mathcal{M}_{\Sigma_{1}} \times \mathcal{M}_{\Sigma_{2}})$ to obtain $(Q,f_{Q})=i^{!CBM}_{*}(P,f_{P}) \in C^{BM}_{*}(\Xi(\mathcal{M}_{\Sigma_{1}} \times \mathcal{M}_{\Sigma_{2}}))$. Then $\Xi^{!CBM}_{*}(P,f_{P})=(Q \times \R^{n},f_{Q \times \R^{n}})=\pi^{!CBM}_{*}(Q,f_{Q})$ is obtained from $(Q,f_{Q})$ by putting a bivalent vertex on each of the edges of $\Gamma$ which originate from attaching an outgoing edge of $\Gamma_{1}$ to an incoming edge of $\Gamma_{2}$. 
\begin{proof}[Proof of Proposition \ref{PropGlueAxiom}] To prove the Proposition, it suffices to show that the following diagram becomes commutative in homology:\\\\
\begin{xy} \hspace{20pt} \xymatrix @C=0.45in{ 
(C^{*}(f))^{\otimes n_{+}} \ar[r]^/-2em/{F^{M}_{\Sigma}} \ar[d]^/0em/{F^{M}_{\Sigma_{1}}} & C^{BM}_{*}(\mathcal{M}_{\Sigma}) \otimes (C^{*}(f))^{\otimes n} \ar[d]^/0em/{\Xi^{!CBM}_{*}\otimes Id} \\ 
  C^{BM}_{*}(\mathcal{M}_{\Sigma_{1}}) \otimes (C^{*}(f))^{\otimes n} \ar[r]^/-1.2em/{Id \otimes F^{M}_{\Sigma_{2}}} & C^{BM}_{*}(\mathcal{M}_{\Sigma_{1}} \times \mathcal{M}_{\Sigma_{2}}) \otimes  (C^{*}(f))^{\otimes n_{-}}. 
} \end{xy} \\ \\ 
We will find a chain homotopy between the compositions $(\Xi^{!CBM}_{*} \otimes Id) \circ F^{M}_{\Sigma}$ and $(Id \otimes F^{M}_{\Sigma_{2}}) \circ F^{M}_{\Sigma_{1}}$.\\ \\ 
Denote $\Phi_{0}=(\Xi^{!CBM}_{*} \otimes Id) \circ F^{M}_{\Sigma}$ and $\Phi_{n}=(Id \otimes F^{M}_{\Sigma_{2}}) \circ F^{M}_{\Sigma_{1}}$. Let us construct cochain maps \begin{equation} \Phi_{1}, \dots, \Phi_{n-1}: (C^{*}(f))^{\otimes n_{+}}  \rightarrow C^{BM}_{*}(\mathcal{M}_{\Sigma_{1}} \times \mathcal{M}_{\Sigma_{2}}) \otimes  (C^{*}(f))^{\otimes n_{-}} \label{EqDefMapsPhi1...} \end{equation}
and show that for $k=0, \dots ,n-1$, there is a chain homotopy between $\Phi_{k}$ and $\Phi_{k+1}$. Denote by $\Sigma(k)$ the surface obtained by gluing $\Sigma_{1}$ to $\Sigma_{2}$ along disjoint closed intervals around the first $n-k$ outgoing marked points on $\Sigma_{1}$ resp. the first $n-k$ incoming marked points on $\Sigma_{2}$. Thus $\Sigma(0)=\Sigma$, while $\Sigma(n)$ is the disjoint union of $\Sigma_{1}$ and $\Sigma_{2}$. \\ \\
We define 
$$F_{k}: (C^{*}(f))^{\otimes n_{+}} \rightarrow C^{BM}_{*}(\mathcal{M}_{\Sigma(k)}) \otimes (C^{*}(f))^{\otimes n_{-}}$$
\begin{equation} {\bf p}_{+} \mapsto \sum \limits_{\Gamma(k),{\bf q}} Z_{\Gamma(k),{\bf x}(k)}(({\bf p}_{+},{\bf q}),({\bf q},{\bf p}_{-})) \otimes {\bf p}_{-} \label{EqDefMapFk}, \end{equation}
where the sum is over all the graphs $\Gamma(k)$ in the ribbon graph decomposition (\ref{EqRibbonGrDec}) of $\Sigma(k)$ and over all tuples ${\bf q} \in Crit(f)^{\times k}$. Attaching the last $k$ outgoing edges of $\Gamma(k)$ to the last $k$ incoming edges yields a ribbon graph $\Gamma$ whose associated surface is $\Sigma$. We will assume that the vector field data ${\bf x}$ and ${\bf x}(k)$ for $\Sigma$ and $\Sigma (k)$ are chosen so that for each $k$ and all $\Gamma_{k}$, ${\bf x}$ and ${\bf x}(k)$ coincide on the complement of these edges.\\ \\ The homomorphism $\Phi_{k}$ is given as 
\begin{equation} \Phi_{k}= (\Xi^{k,!CBM}_{*} \otimes Id) \circ F_{k}, \label{EqConstrMapPhi1...} \end{equation}
where
\begin{equation} \Xi^{k}: \mathcal{M}_{\Sigma_{1}} \times \mathcal{M}_{\Sigma_{2}} \rightarrow \mathcal{M}_{\Sigma(k)} \end{equation}
is the gluing map from (\ref{EqGluinOnMSigma}).
\begin{lemma} \label{LemmaHomotopyPhi1...}
For each $k=0, \dots, n-1$, there is a chain homotopy $\Delta_{k}$ between $\Phi_{k}$ and $\Phi_{k+1}$.
\end{lemma}
\begin{proof}
To simplify terminology, we will refer to a pair consisting of an outgoing edge of $\Gamma_{1}$ and the corresponding incoming edge of $\Gamma_{2}$ as an {\it attachment pair}. \\ \\
The above formal definition of the homomorphism $\Phi_{k}$ means the following. The map $(Id \otimes F^{M}_{\Sigma_{2}}) \circ F_{\Sigma_{1}}$ is given by a count of chains corresponding to pairs of graphs in the ribbon graphs decompositions of $\mathcal{M}_{\Sigma_{1}}$ and $\mathcal{M}_{\Sigma_{2}}$, i. e. \begin{equation} (Id \otimes F^{M}_{\Sigma_{2}}) \circ F^{M}_{\Sigma_{1}}: {\bf p}_{+} \mapsto \sum \limits_{\Gamma_{1},\Gamma_{2},{\bf q},{\bf p}_{-}} Z_{{\Gamma_{1}},{\bf x}}({\bf p}_{+},{\bf q}) \otimes Z_{{\Gamma_{2}},{\bf x}'}({\bf q},{\bf p}_{-}) \otimes {\bf p}_{-}, \label{EqExpressFSigmaSigma'} \end{equation}
where the sum is over ribbon graphs $\Gamma_{1}$ and $\Gamma_{2}$ in the ribbon graph decompositions of $\Sigma_{1}$ and of $\Sigma_{2}$ respectively and over all tuples of critical points $({\bf q},{\bf p}_{-}) \in (Crit(f))^{\times (n+n_{-})}$. The map $\Phi_{k}$ is obtained by counts of geometric chains as on the right-hand of (\ref{EqExpressFSigmaSigma'}), but where instead of a broken trajectory of the gradient flow, to the first $n-k$ attachment pairs, a finite piece of a flow trajectory is associated, namely to these attachment pairs  solutions of (\ref{EqFInt}) are associated, where the parameter $l_{e}$ is given as the sum of the lengths of the two edges of the pair. We denote the spaces of such flow graphs, partially compactified as in Definitioin \ref{DefMGammaxBar} by allowing breaking of trajectories along external edges as well as collapsing of internal edges, by $\overline{\mathcal{M}}_{\Gamma_{1},\Gamma_{2},k}({\bf p}_{-},{\bf q}, {\bf p}_{+})$ and write
\begin{equation} \pi_{\Gamma_{1},\Gamma_{2}}: \overline{\mathcal{M}}_{\Gamma_{1},\Gamma_{2},k}({\bf p}_{-},{\bf q}, {\bf p}_{+}) \rightarrow \mathcal{M}_{\Sigma_{1}} \times \mathcal{M}_{\Sigma_{2}} \label{EqDefPiGammaGamma'} \end{equation} for the projection which maps each graph flow to the underlying metric structures on $\Gamma_{1}$ and on $\Gamma_{2}$. Using this notation, $\Phi_{k}$ is given by 
\begin{equation} \Phi_{k}({\bf p}_{+}) = \sum \limits_{\Gamma_{1},\Gamma_{2},{\bf q},{\bf p}_{-}} Z_{\Gamma_{1}, \Gamma_{2},k}({\bf p}_{+},{\bf q},{\bf p}_{-}) \otimes {\bf p}_{-},  \label{EqExpressionPhi1...} \end{equation} 
where $\Gamma_{1}$ and $\Gamma_{2}$ are as in (\ref{EqExpressFSigmaSigma'}), ${\bf q} \in Crit(f)^{\times k}$ and \begin{equation} Z_{\Gamma_{1}, \Gamma_{2},k}({\bf p}_{+},{\bf q},{\bf p}_{-})=(\overline{\mathcal{M}}_{\Gamma_{1},\Gamma_{2},k}({\bf p}_{-},{\bf q}, {\bf p}_{+}),\pi_{\Gamma_{1},\Gamma_{2}}). \label{EqDefZGammaGamma'} \end{equation}
The construction of a chain homotopy between $\Phi_{k}$ and $\Phi_{k+1}$ relies on the study of spaces $\overline{\mathcal{L}}_{\Gamma_{1},\Gamma_{2},k}({\bf p}_{-},{\bf q}, {\bf p}_{+})$ which we now introduce.\\ \\ Elements of $\overline{\mathcal{L}}_{\Gamma_{1},\Gamma_{2},k}({\bf p}_{-},{\bf q}, {\bf p}_{+})$ are pairs consisting of a positive number $T$ together with a flow graph of the same form as in the case of $\overline{\mathcal{M}}_{\Gamma_{1},\Gamma_{2},k}({\bf p}_{-},{\bf q}, {\bf p}_{+})$, however the equation (\ref{EqFInt}) corresponding to the $(n-k)$-th attachment pair is changed to
\begin{equation} F_{e'}({\textbf{\textit l}},{\bf \gamma})(t)= T\nabla_{g}(f(\gamma_{e'}(t)))+y_{e'} (T,{\textbf{\textit l}},t,\gamma_{e'}(t)), \label{EqDefyTtg}\end{equation}
subject to the following:
\begin{enumerate}
\item $T$ is greater or equal to the sum $l_{e_{1}}+l_{e_{2}}$ of the lengths of the two edges of the attachment pair.
\item For $l_{e_{1}}+l_{e_{2}} \leq T \leq l_{e_{1}}+l_{e_{2}} + 1$, 
\begin{equation} y_{e'}(T,{\textbf{\textit l}}, t,\cdot)=x_{e_{'}}({\textbf{\textit l}},t,\cdot)\label{EqFConditionxe1} 
 \end{equation} for all $t\in[0,1]$.
\item For every $T>l_{e_{1}}+l_{e_{2}}+2$, the vector field $y_{e'} (T,{\textbf{\textit l}},t,\cdot)$ satisfies
\begin{equation} y_{e'} (T,{\textbf{\textit l}},t,\cdot)=\begin{cases} \sigma(Tt) ({\bf x}(k))_{e'_{1}} ({\textbf{\textit l}}_{1},Tt,\cdot) & \text{ for } 0 \leq t \leq \frac{1}{2},\\ \sigma(T(1-t))({\bf x}(k))_{e'_{2}}({\textbf{\textit l}}_{2},Tt,\cdot) & \text{ for } \frac{1}{2} \leq t \leq 1,
\end{cases} \label{EqFConditionxe2} \end{equation}
\end{enumerate} 
where $e'_{1}$ and $e'_{2}$ are the elements of the $(n-k)$th attachment pair and $e'$ the edge obtained by gluing $e'_{1}$ and $e'_{2}$ and erasing the resulting bivalent vertex. Here ${\textbf{\textit l}}_{1}$ and ${\textbf{\textit l}}_{2}$ are the metric structures on $\Gamma_{1}$ and $\Gamma_{2}$ respectively (we also note that the orientations of $e'_{1}$ and $e'_{2}$ define an orientation of $e'$). This space is again partially compactified as in Definition \ref{DefMGammaxBar} by allowing broken trajectories at the edges corresponding to the boundary marked points of $\Sigma(k)$ and collapsing of the remaining edges, but in addition we allow breaking along $e'$ (i. e. we take the union with the spaces $\overline{\mathcal{M}}_{\Gamma_{1},\Gamma_{2},k+1}({\bf p}_{-},{\bf q}', {\bf p}_{+})$ for ${\bf q}' \in (Crit(f))^{\times(k+1)}$). \\ \\
Denoting again by $\pi_{\Gamma_{1},\Gamma_{2}}: \overline{\mathcal{L}}_{\Gamma_{1},\Gamma_{2},k}({\bf p}_{-},{\bf q}, {\bf p}_{+}) \rightarrow \mathcal{M}_{\Sigma_{1}} \times \mathcal{M}_{\Sigma_{2}}$ the natural projection, it follows as in Proposition \ref{PropGener} that $X_{\Gamma_{1},\Gamma_{2},k}({\bf p}_{-},{\bf q}, {\bf p}_{+})=(\overline{\mathcal{L}}_{\Gamma,\Gamma',k}({\bf p}_{-},{\bf q}, {\bf p}_{+}),\pi_{\Gamma_{1},\Gamma_{2}})$ is a well-defined element of $C^{BM}_{*}(\mathcal{M}_{\Sigma_{1}} \times \mathcal{M}_{\Sigma_{2}})$. We denote
$$\Delta_{k}: (C^{*}(f))^{\otimes n_{+}}  \rightarrow C^{BM}_{*}(\mathcal{M}_{\Sigma_{1}} \times \mathcal{M}_{\Sigma_{2}}) \otimes  (C^{*}(f))^{\otimes n_{-}},$$
\begin{equation} {\bf p}_{+} \mapsto \sum \limits_{\Gamma_{1},\Gamma_{2},{\bf q},{\bf p}_{-}} X_{\Gamma_{1}, \Gamma_{2},k}({\bf p}_{+},{\bf q},{\bf p}_{-}) \otimes {\bf p}_{-}, \label{EqDefMapDelta} \end{equation}
where the sum is as in (\ref{EqExpressionPhi1...}). \\ \\
The boundary of $(\overline{\mathcal{L}}_{\Gamma_{1},\Gamma_{2},k}({\bf p}_{-},{\bf q}, {\bf p}_{+}),\pi_{\Gamma_{1},\Gamma_{2}})$ is a disjoint union of components of the following form. Firstly, corresponding to the case $T=l_{e_{1}}+l_{e_{2}}$, we have the components of $\overline{\mathcal{M}}_{\Gamma_{1},\Gamma_{2},k}({\bf p}_{-},{\bf q}, {\bf p}_{+})$. Secondly, corresponding to the case $T \rightarrow \infty$, we have boundary components of the form $\overline{\mathcal{M}}_{\Gamma_{1},\Gamma_{2},k+1}({\bf p}_{-},{\bf q}', {\bf p}_{+})$, where ${\bf q}' \in (Crit(f))^{k+1}$. Thirdly, there are the boundary components corresponding to the breaking of trajectories along the incoming edges of $\Gamma_{1}$ and the outgoing edges of $\Gamma_{2}$. Finally, we have the boundary components corresponding to collapsing an internal edge of $\Gamma_{1}$ or of $\Gamma_{2}$.\\ \\
Summing up all the boundary components of the first two types yields
$$\Phi_{k+1}-\Phi_{k},$$ while the sum of the boundary components of the third type and all the expressions $(\partial X_{\Gamma_{1}, \Gamma_{2},k}({\bf p}_{+},{\bf q},{\bf p}_{-})) \otimes {\bf p}_{-}$ yields $$d \circ \Delta_{k}+\Delta_{k} \circ d.$$ Finally, using Lemma \ref{LemmaProofMain1.1}, the sum of all the components of the last type is zero.
\end{proof}
This completes the Proof of Proposition \ref{PropGlueAxiom}.
\end{proof}

{\footnotesize HUMBOLDT-UNIVERSIT\"AT BERLIN, INSTITUT F\"UR MATHEMATIK,  RUDOWER \\ CHAUSSEE 25, 12489 BERLIN, GERMANY}\\
{\it Email Address:} frommv@math.hu-berlin.de 

\begin{thebibliography}{1000}
\bibitem[AdLeRu-07]{AdLeRu} A. Adem, J. Leida, Y. Ruan, {\it Orbifolds and Stringy Topology}, Cambridge Tracts in Mathematics 171, Cambridge University Pess, Cambridge, 2007.
\bibitem[BeCoh-94]{BeCo} M. Betz, R.L. Cohen, {\it Graph Moduli Spaces and Cohomology Operations}, Turkish J. Math. {\bf 18} (1994), no. 1, 23-41.
\bibitem[BlCohTe-09]{BluCohTel} A. Blumberg, R. Cohen, C. Teleman, {\it Open-Closed Topological Field Theories, String Topology, and Hochschild Homology}, Alpine Perspectives on Algebraic Topology, 53-76, Contemp. Math., 504, Amer. Math. Soc., Providence, RI, 2009.
\bibitem[BoMo-60]{BoMo} A. Borel, J. Moore, {\it Homology Theory for Locally Compact Spaces}, Michigan Math. J. {\bf 7} (1960), 137-159.
\bibitem[CohNor-12]{CoNo} R.L. Cohen, P. Norbury, {\it Morse Field Theory}, Asian J. Math {\bf 16} (2012), no. 4, 661-711.
\bibitem[CoVo-03]{CoVo} J. Conant, K. Vogtmann, {\it On a Theorem of Kontsevich}, Algeb. Geom. Topol. {\bf 3} (2003), 1167-1224.
\bibitem[Cos-07a]{Cos1} K. Costello, {\it Topological Conformal Field Theories and Calabi-Yau Categories}, Adv. Math. {\bf 210} (2007), no. 1, 165-214.
\bibitem[Cos-07b]{Cos2} K. Costello, {\it Topological Conformal Field Theories and Gauge Theories}, Geom. Topol. {\bf 11} (2007), 1539-1579.
\bibitem[Fl-89]{Fl} A. Floer, {\it Witten's Complex and Infinite Dimensional Morse Theory}, J. Differential Geometry {\bf 30} (1989), no. 1, 207-221.
\bibitem[Fu-97]{Fu} K. Fukaya, {\it Morse Homotopy and its Quantization}, Geometric Topology (Athens, GA, 1993), 409-440, AMS/IP Stud. Adv. Math. Vol. 2.1, Amer. Math. Soc., Providence, RI, 1997.
\bibitem[FuOh-98]{FOh} K. Fukaya, Y.-G. Oh, {\it Zero-Loop Open Strings in the Cotangent Bundle and Morse Homotopy}, Asian J. Math. {\bf 1} (1997), no. 1, 96-180.
\bibitem[FOOO-09]{FOOO} K. Fukaya, Y.-G. Oh, H. Ohta and K. Ono, {\it Lagrangian Intersection Floer Theory: Anomaly
and Obstruction. Part I.}, AMS/IP Studies in Advanced Mathematics, 46.1, American Mathematical Society, Providence, RI; International Press, Somerville, MA, 2009.
\bibitem[Ge-94]{Ge} E. Getzler, {\it Batalin-Vilkovisky Algebras and Two-Dimensional Topological Field Theories}, Comm. Math. Phys. {\bf 159} (1994), no. 2, 265-285.
\bibitem[Har-88]{Ha} H. Harer, {\it The Cohomology of the Moduli Space of Curves}, Theory of Moduli (Montecatini Terme, 1985), 138-221, Lecture Notes in Math., 1337, Springer, Berlin, 1988.
\bibitem[HamLaz-08]{HamLaz} A. Hamilton, A. Lazarev, {\it Characteristic Classes of $A_{\infty}$-Algebras}, J. Homotopy Relat. Struct. {\bf 3} (2008), no. 1, 65-111.
\bibitem[Ja-00]{Ja} M. Jakob, {\it An Alternative Approach to Homology}, Une D\'egustation Topologique [Topological Morsels]: Homotopy Theory in the Swiss Alps (Arolla, 1999), 87-97, Contemp. Math., 265, Amer. Math. Soc., Providence, RI, 2000. 
\bibitem[Jo-10]{Jo} D. Joyce, {\it On Manifolds with Corners}, preprint (2010), arxiv: 0910.3518.
\bibitem[Ko-94]{Ko} M. Kontsevich, {\it Feynman Diagrams and Low-dimensional Topology}, First European Congress of Mathematics, Vol. II (Paris, 1992), 97-121, Progr. Math., 120, Birkh\"auser, Basel, 1994.
\bibitem[Pe-87]{Pe} R. Penner, {\it The Decorated Teichm\"uller Space of Punctured Surfaces}, Comm. Math. Phys. {\bf 113} (1987), 299-339.
\bibitem[Sch-93]{Sch} M. Schwarz, {\it Morse Homology}, Progress in Mathematics, 111, Birkh\"auser Verlag, Basel, 1993.
\bibitem[Se-04]{Se} G. Segal, {\it The Definition of Conformal Field Theory}, Topology, Geometry and Quantum Field Theory, 421-577, London Math. Soc. Lecture Notes Ser., 308, Cambridge Univ. Press, Cambridge, 2004.
\bibitem[We-12]{We} K. Wehrheim, {\it Smooth Structures on Morse Trajectory Spaces, Featuring Finite Ends and Associative Gluing}, Preprint, arxiv:1205.0713 (2012). 
\bibitem[Wi-95]{Wi} E. Witten, Chern-Simons Gauge Theory as a String Theory, {\it The Floer Memorial Volume}, 637-678, Progr. Math., 133, Birkh\"auser, Basel, 1995.
\end{thebibliography}
\end{document}